\documentclass{article}

\usepackage{geometry}
\usepackage[normalem]{ulem}
\usepackage{amsfonts}
\usepackage{graphicx}
\usepackage{subfig}
\usepackage{epstopdf}
\usepackage{algorithmic}
\usepackage[ruled,algosection]{algorithm2e}
\ifpdf
  \DeclareGraphicsExtensions{.eps,.pdf,.png,.jpg}
\else
  \DeclareGraphicsExtensions{.eps}
\fi

\usepackage{amsthm}
\newtheorem{theorem}{Theorem}[section]
\newtheorem{lemma}[theorem]{Lemma}
\newtheorem{remark}[theorem]{Remark}

\usepackage{amsmath}
\numberwithin{equation}{section}

\newcommand{\AC}[1]{{\color{black}#1}}
\newcommand{\toberemoved}[1]{}
\newcommand{\TrackChange}[1]{{\color{black}#1}}

\title{Stochastic Primal Dual Hybrid Gradient Algorithm with Adaptive Step-Sizes}

\author{Antonin Chambolle\thanks{CEREMADE, Université Paris-Dauphine, Place du Maréchal De Lattre De Tassigny, 75775 Paris, and MOKAPLAN, INRIA Paris, France.}
\and Claire Delplancke\thanks{EDF Lab Paris-Saclay, route de Saclay, 91300 Palaiseau, France. CD was at the Department of Mathematical Sciences, University of Bath, while the research presented in this article was undertaken.}
\and Matthias J.\@ Ehrhardt\thanks{Department of Mathematical Sciences, University of Bath, Claverton Down, Bath BA2 7AY, United Kingdom.}
\and Carola-Bibiane Sch\"onlieb\thanks{Department of Applied Mathematics and Theoretical Physics, University of Cambridge, Wilberforce Road, Cambridge  CB3 0WA, United Kingdom.}
\and Junqi Tang\thanks{School of Mathematics, University of Birmingham, Edgbaston, Birmingham B15 2TT, United Kingdom.}
}

\usepackage[english]{babel}
\usepackage{enumerate}
\usepackage{amsmath,amsfonts,amssymb}
\usepackage{color}
\usepackage{graphicx}
\usepackage{tabularx}
\usepackage{stmaryrd}                         
\usepackage{multicol}
\usepackage{hyperref}                           

\usepackage{tikz}
\usetikzlibrary{arrows,shapes}

\tikzstyle{every picture}+=[remember picture]
\tikzstyle{na} = [baseline=-.5ex]

  \usetikzlibrary{arrows}
  \usetikzlibrary{automata}
  \usetikzlibrary{spy}
\usepackage{pgfplots}
  \pgfplotsset{width=7cm,compat=1.8}

\begin{document}
\maketitle

\begin{abstract}
In this work we propose a new primal-dual algorithm with adaptive step-sizes. The stochastic primal-dual hybrid gradient (SPDHG) algorithm with constant step-sizes has become widely applied in large-scale convex optimization across many scientific fields due to its scalability. While the product of the primal and dual step-sizes is subject to an upper-bound in order to ensure convergence, the selection of the ratio of the step-sizes is critical in applications. Up-to-now there is no systematic and successful way of selecting the primal and dual step-sizes for SPDHG. In this work, we propose a general class of adaptive SPDHG (A-SPDHG) algorithms, and prove their convergence under weak assumptions. We also propose concrete parameters-updating strategies which satisfy the assumptions of our theory and thereby lead to convergent algorithms. Numerical examples on computed tomography demonstrate the effectiveness of the proposed schemes.
\end{abstract}



\section{Introduction}
The stochastic primal-dual hybrid gradient (SPDHG) algorithm introduced in \cite{SPDHG} is a stochastic version of the primal-dual hybrid gradient (PDHG) algorithm, also known as Chambolle--Pock algorithm \cite{PDHG}. SPDHG has proved more efficient than PDHG for a variety of problems in the framework of large-scale non-smooth convex inverse problems~\cite{SPDHG_PET, CIL, ListmodePET, A-SPDHG_MPI}. Indeed, SPDHG only uses a subset of the data at each iteration, hence reducing the computational cost of evaluating the forward operator and its adjoint; as a result, for the same computational burden, SPDHG attains convergence faster than PDHG. This is especially relevant in the context of medical imaging, where there is a need for algorithms whose convergence speed is compatible with clinical standards, and at the same time able to deal with convex, non-smooth priors like Total Variation (TV), which are well-suited to ill-posed imaging inverse problems, but preclude the recourse to scalable gradient-based methods.

Like PDHG, SPDHG is provably convergent under the assumption that the product of its primal and dual step-sizes is bounded by a constant depending on the problem to solve. On the other hand, the ratio between the primal and dual step-sizes is a free parameter, whose value needs to be chosen by the user. The value of this parameter, which can be interpreted as a control on balance between primal and dual convergence, can have a severe impact on the convergence speed of PDHG, and the same also holds true for SPDHG \cite{SPDHG_primal_dual_balance}. This leads to an important challenge in practice, as there is no known theoretical or empirical rule to guide the choice of the parameter. Manual tuning is computationally expensive, as it would require running and comparing the algorithm on a range \TrackChange{of} values, and there is no guarantee that a value leading to fast convergence for one dataset would keep being a good choice for another dataset. For PDHG, \cite{PDHG_adaptive} have proposed an online primal-dual balancing strategy to solve the issue, where the values of the step-sizes evolve along the iterations. More generally, adaptive step-sizes have been used for PDHG with backtracking in \cite{PDHG_adaptive,PDHG_linesearch}, adapting to local smoothness in \cite{PDHG_local_smoothness} and are widely used for a variety of other algorithms, namely gradient methods in \cite{VariableMetricGD}, subgradient methods in \cite{VariableMetricSubgradient} and splitting methods in \cite{VariableMetricForwardBackwardI,VariableMetricForwardBackwardII,VariableMetricBlockCoordinate,VariableMetricPD,VariableMetricVI} to improve convergence speed and bypass the need for explicit model constants, like Lipschitz constants or operator norms. For SPDHG, an empirical adaptive scheme has been used for Magnetic Particle Imaging but without convergence proof~\cite{A-SPDHG_MPI}.

On the theoretical side, a standard procedure to prove the convergence of proximal-based algorithms for convex optimization is to use the notion of Féjer-monotonicity \cite{BauschkeCombettes}. Constant step-sizes lead to a fixed metric setting, while adaptive step-sizes lead to a variable metric setting.  Work \cite{CombettesVariable} states the convergence of deterministic Féjer-monotone sequences in the variable metric setting, while work \cite{CombettesStochastic} is concerned by the convergence of random Féjer-monotone sequences in the fixed metric setting.

In this work, we introduce and study an adaptive version of SPDHG. More precisely:

\begin{itemize}
    \item We introduce a broad class of strategies to adaptively choose the step-sizes of SPDHG. This class includes, but is not limited to, the adaptive primal-dual balancing strategy, where the ratio of the step-sizes, which controls the balance between convergence of the primal and dual variable, is tuned online.
    
    \item We prove the almost-sure convergence of SPDHG under the schemes of the class. In order to do that, we introduce the concept of $C$-stability, which generalizes the notion of Féjer-monotonicity, and we prove the convergence of random $C$-stable sequences in a variable metric setting, hence generalizing results from \cite{CombettesVariable} and \cite{CombettesStochastic}. We then show that our proposed algorithm falls within this novel theoretical framework by following similar strategies than in the almost-sure convergence proofs of \cite{SPDHG_as,SPDHG_alacaoglu}.
    
    \item We compare the performance of SPDHG for various adaptive schemes and the known fixed step-size scheme on large-scale imaging inverse tasks (sparse-view CT, limited-angle CT, low-dose CT). We observe that the primal-dual balancing adaptive strategy is always as fast or faster than all the other strategies. In particular, it consistently leads to substantial gains in convergence speed over the fixed strategy if the fixed step-sizes, while in the theoretical convergence range, are badly chosen. This is especially relevant as it is impossible to know whether the fixed step-sizes are well or badly chosen without running expensive comparative tests. Even in the cases where the SPDHG's fixed step-sizes are well-tuned, meaning that they are in the range to which the adaptive step-sizes are observed to converge, we observe that our adaptive scheme still provides convergence acceleration over the standard SPDHG after a certain number of iterations. Finally, we pay special attention to the hyperparameters used in the adaptive schemes. These hyperparameters are essentially controlling the degree of adaptivity for the algorithm and each of them has a clear interpretation and is easy to choose in practice. We observe in our extensive numerical tests that the convergence speed of our adaptive scheme is robust to the choices of these parameters within the empirical range we provide, hence can be applied directly to the problem at hand without fine-tuning, and solves the step-sizes choice challenge encountered by the user.
    

    
\end{itemize}

The rest of the paper is organized as follows. In Section \ref{sec:theory}, we introduce SPDHG with adaptive step-sizes, state the convergence theorem, and carry the proof. In Section \ref{sec:AlgorithmicDesign}, we propose concrete schemes to implement the adaptiveness, followed by \TrackChange{numerical} tests on CT data in Section \ref{sec:Numerics}. We conclude in Section \ref{sec:conclusion}. Finally, Section \ref{sec:complements} collects some useful lemmas and proofs.

\section{Theory}
\label{sec:theory}

\subsection{Convergence theorem}

The variational problem to solve takes the form:
\begin{align*}
\min_{x\in X} \sum_{i=1}^n f_i(A_ix) + g(x),
\end{align*}
where $X$ and $(Y_i)_{i \in \left\{1,\dots, n\right\}}$ are Hilbert spaces, $A_i:X\rightarrow Y_i$ are bounded linear operators, $f_i:Y_i\rightarrow \mathbb{R}\cup\left\{+\infty\right\}$ and $g:X\rightarrow \mathbb{R}\cup\left\{+\infty\right\}$ are convex functions. We define $Y=Y_1\times\dots\times Y_n$ with elements $y=(y_1,\dots,y_n)$ and $A:X\rightarrow Y$ such that $Ax = (A_1 x,\dots,A_n x)$. \TrackChange{The associated saddle-point problem reads as
\begin{align}
\min_{x\in X}\, \sup_{y\in Y}\, \sum_{i=1}^n \langle A_ix, y_i \rangle - f_i^*(y_i)  + g(x),
\label{eq:SaddlePointProblem}
\end{align}
\TrackChange{where $f_i^*$ stands for the Fenchel conjugate of $f_i$.}
The set of solution to \eqref{eq:SaddlePointProblem} is denoted by $\mathcal{C}$, the set of non-negative integers by $\mathbb{N}$ and   $\llbracket 1,n \rrbracket$ stands for $\left\{1,\dots, n\right\}$. Elements $(x^*,y^*)$ of  $\mathcal{C}$ are called saddle-points and  characterized by
\begin{align}
A_ix^* \in \partial f_i^*(y_i^*),\TrackChange{\,i\in \llbracket 1,n \rrbracket ; \quad -\sum_{i=1}^n A_i^*y^* \in \partial g(x^*).}
\label{eq:SaddlePointEq}
\end{align}}

\TrackChange{In order to solve the saddle-point problem, we introduce the adaptive stochastic primal-dual hybrid gradient (A-SPDHG) algorithm in Algorithm \ref{alg:A-SPDHG}. At each iteration $k\in \mathbb{N}$, A-SPDHG involves the following five steps:
\begin{itemize}
    \item update the primal step-size $\tau^k$ and the dual step-sizes $(\sigma_i^{k})_{i \in \llbracket 1,n \rrbracket}$ (line \ref{ln:stepsizes_update});
    \item update the primal variable $x^k$ by a proximal step with step-size $\tau^{k+1}$ (line \ref{ln:primal_update});
    \item randomly choose an index $i$ with probability $p_i$ (line \ref{ln:random_choice});
    \item update the dual variable $y_i^k$ by a proximal step with step-size $\sigma_i^{k+1}$ (line \ref{ln:dual_update});
    \item compute the extrapolated dual variable (line \ref{ln:dual_extrapolation}).
\end{itemize}
A-SPDHG is \textit{adaptive} in the sense that the step-sizes values are updated at each iteration according to an update rule which takes into account the value of the primal and dual iterates $x^l$ and $y^l$ up to the current iteration. As the iterates are stochastic, the step-sizes are themselves stochastic, which must be carefully accounted for in the theory.\newline}

\begin{algorithm}[h!]
\caption{A-SPDHG (variable step-sizes, serial sampling)} 
\label{alg:A-SPDHG}
\begin{algorithmic}[1]
\STATE{Input: dual step-sizes $(\sigma_i^0)_{i \in \llbracket 1,n \rrbracket}$, primal step-size $\tau^0$, update rule; probabilities $(p_i)_{i \in \llbracket 1,n \rrbracket}$; primal variable $x^0$, dual variable $y^0$}
\STATE{Initialize $\bar{y}^0=y^0$}
\FOR{$k\in \llbracket 0, K-1 \rrbracket$}
\STATE{Determine $(\sigma_i^{k+1})_{i \in \llbracket 1,n \rrbracket}$, $\tau^{k+1}$ according to the update rule \TrackChange{and the values of $(\sigma_i^{l})_{i \in \llbracket 1,n \rrbracket}$, $\tau^{l}$, $x^l$ and $y^l$ for $l\in \llbracket 0,k \rrbracket$.}} \label{ln:stepsizes_update}
\STATE{$x^{k+1} = \text{prox}_{\tau^{k+1} g}(x^k - \tau^{k+1} A^* \bar{y}^k)$} \label{ln:primal_update}
\STATE{Randomly pick $i \in \llbracket 1,n \rrbracket$ with probability $p_i$} \label{ln:random_choice}
\STATE{$y_j^{k+1} = \begin{cases} 
            \text{prox}_{\sigma_i^{k+1} f_i^*}(y_i^k + \sigma_i^{k+1} A_i x^{k+1}) & \text{if } j=i\\
            y_j^k & \text{if } j\neq i
            \end{cases}$}\label{ln:dual_update}
\STATE{$\bar{y}_j^{k+1} = \begin{cases} 
            y_i^{k+1} + \frac{1}{p_i}\left(y_i^{k+1} - y_i^{k}\right) & \text{if } j=i\\
            y_j^{k} & \text{if } j\neq i
            \end{cases}$}\label{ln:dual_extrapolation}
\ENDFOR
\RETURN $x^K$
\end{algorithmic}
\end{algorithm}

 \TrackChange{Before turning to the convergence of A-SPDHG, let us recall some facts about the state-of-the-art SPDHG. Each iteration of SPDHG involves the selection of a random subset of $\llbracket 1,n \rrbracket$. In the serial sampling case where the random subset is a singleton, SPDHG algorithm  \cite{SPDHG}  is a special case of Algorithm \ref{alg:A-SPDHG} with the update rule 
\begin{align*}
    \begin{cases}
    \sigma_i^{k+1} \,= \, \sigma_i^{k} \,(=\sigma_i),\,{i \in \llbracket 1,n \rrbracket},\\
    \tau^{k+1} \, = \, \tau^k \,(=\tau_i),
    \end{cases}
    \quad k\in \mathbb{N}.
\end{align*}
} 

Under the condition
\begin{align}
\tau \sigma_i < \frac{p_i}{\|A_i\|^2},\quad i \in \llbracket 1,n \rrbracket,
\label{eq:ConvergenceCondition}
\end{align}
SPDHG iterates converge almost surely to a solution of the saddle-point problem \eqref{eq:SaddlePointProblem} \cite{SPDHG_alacaoglu,SPDHG_as}.\newline

\TrackChange{Let us now turn to the convergence of A-SPDHG.} The main theorem, Theorem \ref{thm:A-SPDHG} below, gives conditions on the update rule under which A-SPDHG is provably convergent. Plainly speaking, these conditions are threefold: 
\begin{itemize}
    \item[(i)] the step-sizes for step $k+1$, $(\sigma_i^{k+1})_{i \in \llbracket 1,n \rrbracket}$ and $\tau^{k+1}$, depend only \TrackChange{on} the iterates up to step $k$,
    \item[(ii)] the step-sizes satisfy a uniform version of condition \eqref{eq:ConvergenceCondition},
     \item[(iii)] the step-sizes sequences $(\tau^k)_{k \geq 0}$ and $(\sigma_i^k)_{k \geq 0}$ for $i \in \llbracket 1,n \rrbracket$ do not decrease too fast. More precisely, they are uniformly almost surely quasi-increasing in the sense defined below.\newline
\end{itemize}

In order to state the theorem rigorously, let us introduce some useful notation and definitions. For all $k\in \mathbb{N}$, the $\sigma$-algebra generated by the iterates up to point $k$, $\mathcal{F}\left((x^l,y^l), l \in \llbracket 0,k \rrbracket \right)$, is denoted by $\mathcal{F}^k$. We say that a sequence $(u^k)_{k\in \mathbb{N}}$ is $\left(\mathcal{F}^k\right)_{k\in \mathbb{N}}$-adapted if for all $k\in\mathbb{N}$, $u^k$ is measurable with respect to $\mathcal{F}^k$. 

A positive real sequence $(u^k)_{k \in \mathbb{N}}$ is said to be \textit{quasi-increasing}  if there exists a sequence $(\eta^k)_{k \in \mathbb{N}}$ with values in $[0,1)$, called the control on $(u^k)_{k \in \mathbb{N}}$,  such that $\sum_{k=1}^\infty \eta^k < \infty$ and :
\begin{align}
u^{k+1} &\geq (1-\eta^k) u^k,\,\quad k \in \mathbb{N}.\label{eq:QI}
\end{align}

By extension, we call a random positive real sequence $(u^k)_{k \in \mathbb{N}}$ \textit{uniformly almost surely quasi-increasing}  if there exists a deterministic sequence $(\eta^k)_{k \in \mathbb{N}}$ with values in $[0,1)$ such that $\sum_{k=1}^\infty \TrackChange{\eta^k} < \infty$ and equation \eqref{eq:QI} above holds almost surely (a.s.\@). 

\begin{theorem}[Convergence of A-SPDHG]
\label{thm:A-SPDHG}
Let $X$ and $Y$ be separable Hilbert spaces, $A_i:X\rightarrow Y_i$ bounded linear operators, $f_i:Y_i\rightarrow \mathbb{R}\cup\left\{+\infty\right\}$ and $g:X\rightarrow \mathbb{R}\cup\left\{+\infty\right\}$ proper, convex and lower semi-continuous functions for all $i \in \llbracket 1,n \rrbracket$. Assume that the set of saddle-points $\mathcal{C}$ is non-empty and the sampling is proper, that is to say $p_i>0$ for all $i\in\llbracket 1,n \rrbracket$. If the following conditions are met:
\begin{itemize}
    \item[(i)] the step-size sequences $(\tau^{k+1})_{k \in \mathbb{N}}, (\sigma_i^{k+1})_{k \in \mathbb{N}},\,i \in \llbracket 1,n \rrbracket$ are $\left(\mathcal{F}^k\right)_{k \in \mathbb{N}}$-adapted,
    \item[(ii)] there exists $\beta\in(0,1)$ such that for all \TrackChange{indices $i \in \llbracket 1,n \rrbracket$ and iterates $k \in \mathbb{N}$},
    \begin{align}
    \tau^k \sigma_i^k  \frac{\|A_i\|^2}{p_i} \leq \beta < 1,\label{eq:A-ConvergenceCondition}
    \end{align}
    \item [(iii)] the initial step-sizes $\tau^0$ and $\sigma_i^{0}$ for all \TrackChange{indices} $i \in \llbracket 1,n \rrbracket$ are positive and the step-sizes sequences $(\tau^{k})_{k \in \mathbb{N}}$ and $(\sigma_i^{k})_{k \in \mathbb{N}}$ for all \TrackChange{indices} $i \in \llbracket 1,n \rrbracket$ are uniformly almost surely quasi-increasing,
\end{itemize}
then the sequence of iterates $(x^k,y^k)_{k \in \mathbb{N}}$ converges almost surely to an element of $\mathcal{C}$.
\end{theorem}

While the conditions (i)-(iii) are general enough to cover a large range of step-sizes update rules, we will focus in practice on the primal-dual balancing strategy, which consists in scaling the primal and the dual step-sizes by an inverse factor at each iteration. In that case, the update rule depends on a random positive sequence $(\gamma^k)_{k\in \mathbb{N}}$ and reads as:
\begin{align}
\tau^{k+1} &= \frac{\tau^k}{\gamma^{k}} ,\quad \sigma_i^{k+1}= \gamma^k \sigma_i^k, \quad i \in \llbracket 1,n \rrbracket.
\label{eq:UpdateRule}
\end{align}

\begin{lemma}[Primal-dual balancing]
Let the step-sizes sequences satisfy equation \eqref{eq:UpdateRule} and assume in addition that $(\gamma^k)_{k\in \mathbb{N}}$ is $\left(\mathcal{F}^k\right)_{k \in \mathbb{N}}$-adapted, that the initial step-sizes satisfy
\begin{align*}
     \tau^0 \sigma_i^0  \frac{\|A_i\|^2}{p_i}  < 1,\quad i \in \llbracket 1,n \rrbracket,
\end{align*}
and are positive, that there exists a deterministic sequence $(\epsilon^k)_{k\in \mathbb{N}}$ with values in $[0,1)$ such that $\sum \epsilon^k < \infty$ and for all $k\in \mathbb{N}$ and $i \in \llbracket 1,n \rrbracket$,
\begin{align}
   \min \left\{ \gamma^k, (\gamma^k)^{-1} \right\} &\geq 1 - \epsilon^k.
\label{eq:PrimalDualBalancing}
\end{align}
Then, the step-sizes sequences satisfy assumptions (i)-(iii) of Theorem \ref{thm:A-SPDHG}.
\label{lm:PrimalDualBalancing}
\end{lemma}

Lemma \ref{lm:PrimalDualBalancing} is proved in Section \ref{sec:complements}. \\

\textit{Connection with the literature:}
\begin{itemize}
    \item The primal-dual balancing strategy has been introduced in \cite{PDHG_adaptive} for PDHG and indeed for $n=1$ we recover with Lemma \ref{lm:PrimalDualBalancing} the non-backtracking algorithm presented in \cite{PDHG_adaptive}. As a consequence, our theorem also implies the pointwise convergence of this algorithm, whose convergence was established in the sense of vanishing residuals in \cite{PDHG_adaptive}.
    \item  Still for PDHG, \cite{PDHG_linesearch} proposes without proof an update rule where the ratio of the step-sizes is either quasi non-increasing or quasi non-decreasing. This requirement is similar to but not directly connected with ours, where we ask the step-sizes themselves to be quasi non-increasing.
    \item For SPDHG, the angular constraint step-size rule proposed without convergence proof in \cite{A-SPDHG_MPI} satisfies assumptions (i)-(iii).\\
\end{itemize}

\emph{Outline of the proof:} 
Theorem \ref{thm:A-SPDHG} is proved in the following sub-sections. We first define in Section \ref{subsec:VariableMetrics} metrics related to the algorithm step-sizes on the primal-dual product space. As the step-sizes are adaptive, we obtain a sequence of metrics. The proof of Theorem \ref{thm:A-SPDHG} is then similar in strategy to those of \cite{SPDHG_alacaoglu} and \cite{SPDHG_as} but requires novel elements to deal with the metrics variability. In Theorem \ref{thm:CStableSequencesConvergence}, we state convergence conditions for an abstract random sequence in a Hilbert space equipped with random variable metrics. In Section \ref{subsec:A-SPDHGIsCStable} and Section \ref{subsec:ClusterPointsAreSaddlePoints} we show that A-SPDHG falls within the scope of Theorem \ref{thm:CStableSequencesConvergence}. We collect all elements and conclude the proof in Section \ref{subsec:ThmProof}.



\subsection{Variable metrics}
\label{subsec:VariableMetrics}

For a Hilbert space $H$, we call $\mathbb{S}(H)$ the set of bounded self-adjoint linear operators from $H$ to $H$, and for all $M \in \mathbb{S}(H)$ we introduce the notation:
\begin{align*}
\|u\|_M^2 &= \langle Mu,u \rangle,\quad u \in H.
\end{align*}
By an abuse of notation we write $\|\cdot\|_\alpha^2 =\|\cdot\|_{\alpha \text{Id}}^2$ for a scalar $\alpha \in \mathbb R$. Notice that $\|\cdot\|_M$ is a norm on $H$ if $M$ is positive definite. Furthermore, we introduce the partial order $\preccurlyeq$ on $\mathbb{S}(H)$ such that for $M,\,N \in \mathbb{S}(H)$, 
\begin{align*}
N \preccurlyeq M \quad \text{if} \quad \forall u \in H,\, \|u\|_N \leq  \|u\|_M.
\end{align*}
We call $\mathbb{S}_\alpha(H)$ the subset of $\mathbb{S}(H)$ comprised of $M$ such that $\alpha \text{Id} \preccurlyeq M$. Furthermore a random sequence $(M^k)_{k \in \mathbb{N}}$ in $\mathbb{S}(H)$ is said to be \textit{uniformly almost surely quasi-decreasing} if there exists a \TrackChange{deterministic} non-negative sequence $(\eta^k)_{k \in \mathbb{N}}$ such that $\sum_{k=1}^\infty \TrackChange{\eta^k} < \infty$ and a.s.
\begin{align*}
M^{k+1} \preccurlyeq (1+\eta^k) M^k,\,\quad k \in \mathbb{N}.
\end{align*}

Coming back to A-SPDHG, let us define for every iteration $k \in \mathbb{N}$ and every index $i \in \llbracket 1,n \rrbracket$ two block operators of $\mathbb{S}(X\times Y_i)$ as:
\begin{align*}
M_i^k &= \begin{pmatrix} \frac{1}{\tau^k} \text{Id} && -\frac{1}{p_i}\TrackChange{A_i^*} \\ &&\\ -\frac{1}{p_i}\TrackChange{A_i}  && \frac{1}{p_i\sigma_{i}^k} \text{Id}\end{pmatrix}, \quad N_i^k = \begin{pmatrix} \frac{1}{\tau^k} \text{Id} && 0 \\ &&\\ 0 && \frac{1}{p_i\sigma_{i}^k} \text{Id} \end{pmatrix},
\end{align*}
and a block operator of $\mathbb{S}(X\times Y)$ as:
\TrackChange{
\begin{align}
N^k = \begin{pmatrix} \frac{1}{\tau^k} \text{Id} &&&&&&  (0)\\ &\frac{1}{p_1\sigma_{1}^k}\text{Id}&&&&&\\ &&\ddots&&&&\\  &&& \frac{1}{p_i\sigma_{i}^k} \text{Id} &&\\ &&&&&\ddots&\\ (0)&&&&&&\frac{1}{p_n\sigma_{n}^k}\text{Id}\\ \end{pmatrix}.
\label{eq:MetricDefinition}
\end{align}
}
The following lemma translates assumptions (i)-(iii) of Theorem \ref{thm:A-SPDHG} on properties on the variable metric sequences.

\begin{lemma}[Variable metric properties]
\begin{itemize}
    \item[(a)] Assumption (i) of Theorem \ref{thm:A-SPDHG} implies that \TrackChange{$(M_i^{k+1})_{k \in \mathbb{N}}$, $(N_i^{k+1})_{k \in \mathbb{N}},\,i \in \llbracket 1,n \rrbracket$ and $(N^{k+1})_{k \in \mathbb{N}}$} are $\left(\mathcal{F}^k\right)_{k \in \mathbb{N}}$-adapted.
    \item[(b)] Assumption (ii) of Theorem \ref{thm:A-SPDHG} is equivalent to the the existence of $\beta\in(0,1)$ such that for all \TrackChange{indices $i \in \llbracket 1,n \rrbracket$ and iterates $k \in \mathbb{N}$},
    \begin{align*}
    (1-\TrackChange{\sqrt{\beta}}) N_i^k \preccurlyeq M_i^k.
    \end{align*}
    \item[(c)] Assumptions (ii) and (iii) of Theorem \ref{thm:A-SPDHG} imply that $(M_i^k)_{k \in \mathbb{N}}, (N_i^k)_{k \in \mathbb{N}},\,i \in \llbracket 1,n \rrbracket$ and $(N^k)_{k \in \mathbb{N}}$ are uniformly a.s.\@ quasi-decreasing.
    \item[(d)] Assumption (ii) and (iii) of Theorem \ref{thm:A-SPDHG} imply that \TrackChange{the sequences $(\tau^{k})_{k \in \mathbb{N}}$ and $(\sigma_i^{k})_{k \in \mathbb{N}}$ for all $i \in \llbracket 1,n \rrbracket$ are a.s.\@ \TrackChange{bounded from} above and by below by positive constants. In particular, this implies that there exists $\alpha>0$ such that $N_i^k \in \mathbb{S}_\alpha(X\times Y_i)$ for all $i \in \llbracket 1,n \rrbracket$ and $k \in \mathbb{N}$, or equivalently that $N^k \in \mathbb{S}_\alpha(X \times Y)$ for all $k \in \mathbb{N}$.}
\end{itemize}
\label{lm:VariableMetrics}
\end{lemma}

\begin{remark}[Step-sizes induced metrics on the primal-dual product space]
The lemma implies that $M_i^k$, $N_i^k$ and $N^k$ are positive definite, hence induce a metric on the corresponding spaces. If $n=1$ and for constant step-sizes, $M_i^k$ corresponds to the metric used in \cite{PDHG_as_PPA}, where PDHG is reformulated as a proximal point algorithm for a non-trivial metric on the primal-dual product space.
\end{remark}

\begin{proof}[Proof of Lemma \ref{lm:VariableMetrics}]
Assertion (a) of the lemma follows from the fact that for all  iterate $k \in \mathbb{N}$, \TrackChange{the operators $M_i^{k+1}$, $N_i^{k+1}$ and $N^{k+1}$ are in the $\sigma$-algebra generated by $\left\{\tau^{k+1},\, \sigma_i^{k+1}, \, i \in \llbracket 1,n \rrbracket \right\}$}. Assertion (b) follows from equation \eqref{eq:A2C_b} of Lemma \ref{lm:Abstract2Concrete} to be found in the complementary material. The proof of assertion (c) is a bit more involved. Let us assume that assumption (iii) of Theorem \ref{thm:A-SPDHG} holds and let $(\eta_0^k)_{k \in \mathbb{N}}$ and $(\eta_i^k)_{k \in \mathbb{N}}$ be the controls of $(\tau^k)_{k \in \mathbb{N}}$ and $(\sigma_i^k)_{k \in \mathbb{N}}$ for $i \in \llbracket 1,n \rrbracket$ respectively. \TrackChange{We define the sequence $(\eta^k)_{k \in \mathbb{N}}$ by:
\begin{align}
\eta^k &= \max \left\{ \eta_i^k, \, i \in \llbracket 0,n \rrbracket \right\} ,\quad k \in \mathbb{N},
\label{eq:CommonControlOnStepSizes}
\end{align}
which is a common control on $(\tau^k)_{k \in \mathbb{N}}$ and $(\sigma_i^k)_{k \in \mathbb{N}}$ for $i \in \llbracket 1,n \rrbracket$ as the maximum of a finite number of controls.}
Let us fix $k \in \mathbb{N}$ and $i \in \llbracket 1,n \rrbracket$. \TrackChange{Because the intersection of a finite number of measurable events of probability one is again a measurable event of probability one,} it holds almost surely that for all $(x,y_i)\in X \times Y_i$,
\begin{align*}
    \|\TrackChange{(x,y_i)}\|_{N_i^{k+1}}^2 &= \frac{1}{\tau^{k+1}} \|x\|^2 + \frac{1}{\TrackChange{p_i}\sigma_i^{k+1}} \| y_i\|^2 \\
    &\leq \frac{1}{1-\eta^k}  \left(\frac{1}{\tau^{k}} \|x\|^2 + \frac{1}{\TrackChange{p_i}\sigma_i^{k}} \| y_i\|^2\right) \\
    &= \TrackChange{\left(1+\frac{\eta_k}{1-\eta^k}\right)}\|\TrackChange{(x,y_i)}\|_{N_i^{k}}^2.
\end{align*}
Hence the sequence $(N_i^k)_{k \in \mathbb{N}}$ is uniformly quasi-decreasing with control \TrackChange{$\left(\eta^k(1-\eta^k)^{-1}\right)_{k \in \mathbb{N}}$}, which is indeed a positive sequence with bounded sum. \TrackChange{(To see that $\left(\eta^k(1-\eta^k)^{-1}\right)_{k \in \mathbb{N}}$ has a bounded sum, consider that $\left(\eta^k\right)_{k \in \mathbb{N}}$ is summable, hence converges to $0$, hence is smaller than $1/2$ for all integers $k$ bigger than a certain $K$; in turn, for all integers $k$ bigger than $K$, the term $\eta^k(1-\eta^k)^{-1}$ is \TrackChange{bounded from below by $0$ and from above} by $2\eta^k$, hence is summable.)} One can see by a similar proof that  $(N^k)_{k \in \mathbb{N}}$ is uniformly quasi-decreasing with the same control. 
\AC{To follow with the case of $(M_i^k)_{k \in \mathbb{N}}$, we have, as before:
\begin{equation*}
M_i^{k+1} = \begin{pmatrix} \frac{1}{\tau^{k+1}} \text{Id} && -\frac{1}{p_i}\TrackChange{A_i^*} \\ &&\\ -\frac{1}{p_i}\TrackChange{A_i}  && \frac{1}{p_i\sigma_{i}^{k+1}} \text{Id}\end{pmatrix}
\preccurlyeq M_i^k + \frac{\eta^k}{1-\eta^k} N_i^k
\preccurlyeq \left(1+\frac{\eta^k}{1-\eta^k}\frac{1}{1-\sqrt{\beta}}\right) M_i^k
\end{equation*}
thanks to (b).
}

\TrackChange{Let us conclude with the proof of assertion \TrackChange{(d)}. By assumption (iii), the sequences $(\tau^k)_{k \in \mathbb{N}}$ and $(\sigma_i^k)_{k \in \mathbb{N}}$ are uniformly a.s.\@ quasi-increasing. We define a common control $(\eta^k)_{k \in \mathbb{N}}$ as in \eqref{eq:CommonControlOnStepSizes}. Then, the sequences $(\tau^k)_{k \in \mathbb{N}}$ and $(\sigma_i^k)_{k \in \mathbb{N}}$ are a.s.\@ \TrackChange{bounded from} below by the same deterministic constant $C=\min\left\{\tau^0, \,\sigma_i^0,\, i \in \llbracket 1,n \rrbracket \right\} \prod_{j=\TrackChange{0}}^\infty(1-\eta^j)$ which is positive as the initial step-sizes are positive and $(\eta^k)_{k \in \mathbb{N}}$ takes values in $[0,1)$ and has finite sum. Furthermore, by assumption (ii), the product of the sequences $(\tau^k)_{k \in \mathbb{N}}$ and $(\sigma_i^k)_{k \in \mathbb{N}}$ is almost surely \TrackChange{bounded from} above. As a consequence, each sequence $(\tau^k)_{k \in \mathbb{N}}$ and $(\sigma_i^k)_{k \in \mathbb{N}}$ is a.s.\@ \TrackChange{bounded from} above. The equivalence with $N_i^k \in \mathbb{S}_\alpha(X\times Y_i)$ for all $i \in \llbracket 1,n \rrbracket$, and with $N^k \in \mathbb{S}_\alpha(X\times Y)$, is straightforward.}
\end{proof}

\subsection{Convergence of random $C$-stable sequences in random variable metrics}
\label{subsec:CStableSequencesConverge}

Let $H$ be a Hilbert space and $C\subset H$ a subset of $H$. Let $\left(\Omega, \sigma(\Omega), \mathbb{P}\right)$ be a probability space. All random variables in the following are assumed to be defined on $\Omega$ and measurable with respect to $\sigma(\Omega)$ unless stated otherwise. Let $(Q^k)_{k \in \mathbb{N}}$ be a random sequence of $\mathbb{S}(H)$.\newline

A random sequence $(u^k)_{k \in \mathbb{N}}$ with values in $H$ is said to be \textit{stable with respect to the target $C$ relative to $(Q^k)_{k \in \mathbb{N}}$} if for all $u\in C$, the sequence $\left(\|u^k - u\|_{Q^k}\right)_{k \in \mathbb{N}}$ converges almost surely. The following theorem then states sufficient conditions for the convergence of such sequences.

\begin{theorem}[Convergence of $C$-stable sequences]
Let $H$ be a separable Hilbert space, $C$ a closed non-empty subset of $H$, $(Q^k)_{k \in \mathbb{N}}$ a random sequence of $\mathbb{S}(H)$, and $(u^k)_{k \in \mathbb{N}}$ a random sequence of $H$. If the following conditions are met:
\begin{itemize}
    \item[(i)] $(Q^k)_{k \in \mathbb{N}}$ takes values in $\mathbb{S}_\alpha(H)$ for a given $\alpha>0$ and is uniformly a.s.\@ quasi-decreasing,
    \item[(ii)] $(u^k)_{k \in \mathbb{N}}$ is stable with respect to the target $C$ relative to $(Q^k)_{k \in \mathbb{N}}$,
    \item[(iii)] every weak sequential cluster point of $(u^k)_{k \in \mathbb{N}}$ is almost surely in $C$, \TrackChange{meaning that there exists $\Omega_{(iii)}$ a measurable subset of $\Omega$ of probability one such that for all $\omega \in \Omega$, every weak sequential cluster point of $(u^k(\omega))_{k \in \mathbb{N}}$ is in $C$.}
\end{itemize}
then $(u^k)_{k \in \mathbb{N}}$ converges almost surely weakly to a random variable in $C$.
\label{thm:CStableSequencesConvergence}
\end{theorem}

Stability with respect to a target set $C$ is implied by Féjer and quasi-Féjer monotonicity with respect to $C$, which have been studied either for random sequences \cite{CombettesStochastic} or in the framework of variable metrics \cite{CombettesVariable}, but to the best of our knowledge not both at the same time. The proof of Theorem \ref{thm:CStableSequencesConvergence} follows the same lines than \cite[Proposition 2.3 (iii)]{CombettesStochastic} and uses two results from \cite{CombettesVariable}.

\begin{proof}
\TrackChange{The set $C$ is a subset of the separable Hilbert space $H$, hence is separable. As $C$ is a closed and separable, there exists $\left\{ c^n,\, n\in \mathbb{N}\right\}$ a countable \TrackChange{subset of $C$} whose closure is equal to $C$.} Thanks to assumption (ii), there exists for all $n\in \mathbb{N}$ a measurable subset $\Omega_{(ii)}^n$ of $\Omega$ with probability one such that the sequence $\TrackChange{(\|u^k(\omega) - c^n\|_{Q^k(\omega)})_{k \in \mathbb{N}}}$ converges for all $\omega \in \Omega_{(ii)}^n$ . Furthermore, \TrackChange{let $\Omega_{(i)}$ be a measurable subset of $\Omega$ of probability one corresponding to the almost sure property for assumption (i)}. Let
\begin{align*}
\tilde{\Omega} &= \left( \bigcap_{n\geq 0} \Omega_{(ii)}^n \right) \bigcap \Omega_{(i)} \bigcap \Omega_{(iii)}.
\end{align*}
As the intersection of a countable number of measurable subsets of probability one, $\tilde{\Omega}$ is itself a measurable set of $\Omega$ with $\mathbb{P}(\tilde{\Omega})=1$. Fix $\omega \in \tilde{\Omega}$ for the rest of the proof. \newline

The sequence $(Q^k(\omega))_{k \in \mathbb{N}}$ takes values in $\mathbb{S}_\alpha(H)$ for $\alpha>0$ and is quasi-decreasing with control $(\eta^k(\omega))_{k \in \mathbb{N}}$. Furthermore, for all $k \in \mathbb{N}$,
\begin{align*}
\|Q^k(\omega) \| &\leq \left(\prod_{j=0}^{k-1}\left(1+\eta^j\right)\right) \|Q^0(\omega) \| \leq \left(\prod_{j=0}^{\infty}\left(1+\eta^j\right)\right) \|Q^0(\omega) \|,
\end{align*}
where the product $\prod_{j=0}^{\infty}\left(1+\eta^j\right)$ is finite because $(\eta^k)_{k \in \mathbb{N}}$ is positive and summable. By \cite[Lemma 2.3]{CombettesVariable}, $(Q^k(\omega))_{k \in \mathbb{N}}$ converges pointwise strongly to some $Q(\omega)\in \mathbb{S}_\alpha(H)$.\newline

Furthermore, for all $x\in C$, there exists a sequence $(x^n)_{n\in \mathbb{N}}$ with values in $\left\{ c^n,\, n\in \mathbb{N}\right\}$ converging strongly to $x$. By assumption, for all $n\in \mathbb{N}$, the sequence $(\|u^k(\omega) - x^n\|_{Q^k(\omega)})_{k \in \mathbb{N}}$ converges to a limit which shall be called $l^n(\omega)$. For all $n\in \mathbb{N}$ and $k \in \mathbb{N}$, we can write thanks to the triangular inequality:
\begin{align*}
-\|x^n - x\|_{Q^k(\omega)} &\leq \|u^k(\omega) -x\|_{Q^k(\omega)} - \|u^k(\omega) -x^n\|_{Q^k(\omega)} \leq \|x^n - x\|_{Q^k(\omega)}.
\end{align*}
By taking the limit $k\rightarrow +\infty$, it follows that: 
\begin{align*}
-\|x^n - x\|_{Q(\omega)} &\leq \underset{k\rightarrow \infty}{\lim\inf}\, \|u^k(\omega) -x\|_{Q^k(\omega)} - l^n(\omega)\\
&\leq \underset{k\rightarrow \infty}{\lim\sup}\, \|u^k(\omega) -x\|_{Q^k(\omega)} - l^n(\omega) \leq \|x^n - x\|_{Q(\omega)}.
\end{align*}
Taking now the limit $n\rightarrow +\infty$ shows that the sequence $(\|u^k(\omega) - x\|_{Q^k(\omega)})_{k \in \mathbb{N}}$ converges for all $x\in C$. On the other hand, because $\omega \in \Omega_{(iii)}$, the weak cluster points of $(u^k(\omega))_{k \in \mathbb{N}}$ lie in $C$. Hence, by \cite[Theorem 3.3]{CombettesVariable}, the sequence $(u^k(\omega))_{k \in \mathbb{N}}$ converges almost surely to a point $u(\omega)\in C$.
\end{proof}

We are now equipped to prove Theorem \ref{thm:A-SPDHG}. We show in Section \ref{subsec:A-SPDHGIsCStable} and Section \ref{subsec:ClusterPointsAreSaddlePoints} that A-SPDHG satisfies points (ii) and (iii) of Theorem \ref{thm:CStableSequencesConvergence} respectively and conclude the proof in Section \ref{subsec:ThmProof}. Interestingly, the proofs of point (ii) and of point (iii) rely on two different ways of apprehending A-SPDHG. Point (ii) relies on a convex optimisation argument: by taking advantage of the measurability of the primal variable at step $k+1$ with respect to $\mathcal{F}^k$, one can write a contraction-type inequality relating the conditional expectation of the iterates' norm at step $k+1$ to the iterates' norm at step $k$. Point (iii) relies on monotone operator theory: we use the fact that the update from the half-shifted iterations $(y^k,x^{k+1})$ to $(y^{k+1},x^{k+2})$ can be interpreted as a step of a proximal-point algorithm on $X\times Y_i$ conditionally to $i$ being the index randomly selected at step $k$.

\subsection{A-SPDHG is stable with respect to the set of saddle-points}
\label{subsec:A-SPDHGIsCStable}

In this section, we show that $(x^k,y^k)_{k \in \mathbb{N}}$ is stable with respect to $\mathcal{C}$ relative to the variable metrics sequence $(N^k)_{k \in \mathbb{N}}$ defined in equation \eqref{eq:MetricDefinition} above. We introduce the operators $P\in \mathbb{S}(Y)$ and $\Sigma^k\in \mathbb{S}(Y)$ defined respectively by 
\begin{align*}
(Py)_i&=p_iy_i, \quad (\Sigma^ky)_i=\sigma_i^{k}y_i, \quad i \in \llbracket 1,n \rrbracket,
\end{align*}
and the functionals $(U^k)_{k \in \mathbb{N}},\, (V^k)_{k \in \mathbb{N}}$ defined for all $(x,y)\in X\times Y$ as:
\begin{align*}
U^k(y) &= \|y\|_{(P\Sigma^k)^{-1}}^2,\\
V^k(x,y) &= \|x\|_{(\tau^k)^{-1}}^2 - 2\langle P^{-1}Ax, y \rangle + \|y\|_{(P\Sigma^k)^{-1}}^2.
\end{align*}

We begin by recalling the cornerstone inequality satisfied by the iterates of SPDHG stated first in \cite{SPDHG} and reformulated in \cite{SPDHG_alacaoglu}.
\begin{lemma}[\cite{SPDHG_alacaoglu}, Lemma 4.1]
For every saddle-point $(x^*,y^*)$, it a.s.\@ stands that for all $k \in \mathbb{N}\setminus\left\{0\right\}$,
\begin{align}
&\mathbb{E}\left[ V^{k+1}(x^{k+1}-x^*,y^{k+1}-y^k) + U^{k+1}(y^{k+1}-y^*) | \mathcal{F}^k \right]\nonumber\\
\leq & \quad V^{k+1}(x^{k}-x^*,y^{k}-y^{k-1}) + U^{k+1}(y^{k}-y^*)\label{eq:CornerStone} \\
& \quad- V^{k+1}(x^{k+1}-x^k,y^{k}-y^{k-1}).\nonumber
\end{align}
\end{lemma}

The second step is to relate the assumptions of Theorem \ref{thm:A-SPDHG} to properties of the functionals appearing in \eqref{eq:CornerStone}. Let us introduce $Y_{\text{sparse}} \subset Y$ the set of elements $(y_1, \dots, y_n)$ having at most one non-vanishing component. 

\begin{lemma}[Properties of functionals of interest]
Under the assumptions of Theorem \ref{thm:A-SPDHG}, there exists a non-negative, summable sequence $(\eta^k)_{k \in \mathbb{N}}$ such that a.s.\@ for every iterate $k \in \mathbb{N}$ and $x\in X,\,y\in Y,\, z\in Y_{\text{sparse}}$:
\begin{subequations}
\begin{align}
U^{k+1}(y) &\leq (1+\eta^k) U^k(y), \label{eq:FIa} \\
V^{k+1}(x,z) &\leq (1+\eta^k) V^k(x,z), \label{eq:FIb} \\
\|(x,z)\|_{N^k}^2 & \geq \alpha \|(x,z)\|^2, \label{eq:FIe}\\
V^k(x,z) &\geq (1-\beta) \|(x,z)\|_{N^k}^2, \label{eq:FIc}\\
\left|\left\langle P^{-1}A x, z\right\rangle\right| &\leq \TrackChange{\sqrt{\beta} \|x\|_{(\tau^{k})^{-1}}\|z\|_{(P\Sigma^{k})^{-1}}.} \label{eq:FIf}
\end{align}
\end{subequations}
\end{lemma}

\begin{proof}
Let $(\eta_i^k)_{k \in \mathbb{N}}$ and $(\tilde{\eta}_i^k)_{k \in \mathbb{N}}$ be the controls of $(M_i^k)_{k \in \mathbb{N}}$ and $(N_i^k)_{k \in \mathbb{N}}$ respectively for all $i \in \llbracket 1,n \rrbracket$. We define the common control $(\eta^k)_{k \in \mathbb{N}}$ by:
\begin{align}
\eta^k &= \max \left\{ \max \left\{\eta_i^k,\tilde{\eta}_i^k \right\}, i \in \llbracket 1,n \rrbracket\right\} ,\quad k \in \mathbb{N}.
\label{eq:CommonControl}
\end{align}
For all $y\in Y$, we can write
\begin{align*}
U^{k+1}(y) &= \sum_{i=1}^n\|(0,y_i)\|_{N_i^{k+1}}^2 \leq (1+\eta^k) \sum_{i=1}^n\|(0,y_i)\|_{N_i^{k}}^2 = (1+\eta^k) U^k(y),
\end{align*}
which proves \eqref{eq:FIa}. Let us now fix $x\in X$, $z\in Y_{\text{sparse}}$ and $k \in \mathbb{N}$. By definition, there exists $i \in \llbracket 1,n \rrbracket$ such that $z_j=0$ for all $j\neq i$. We obtain the inequalities \eqref{eq:FIb}-\eqref{eq:FIc} by writing:
\begin{align*}
V^{k+1}(x,z) &= \|(x,z_i)\|_{M_i^{k+1}}^2 \leq (1+\eta^k) \|(x,z_i)\|_{M_i^k}^2 = (1+\eta^k) V^{k}(x,z),\\
\|(x,z)\|_{N^k}^2 &= \|(x,z_i)\|_{N_i^k}^2 \geq \alpha \|(x,z_i)\|^2 = \TrackChange{\alpha} \|(x,z)\|^2, \\
V^{k}(x,z) &= \|(x,z_i)\|_{M_i^{k}}^2 \geq (1-\beta)\|(x,z_i)\|_{N_i^k}^2 = (1-\beta) \|(x,z)\|_{N^k}^2.
\end{align*}

\TrackChange{
Finally, we obtain inequality \eqref{eq:FIf} by writing:
\begin{align*}
\left|\left\langle P^{-1}A x, z\right\rangle\right| &= \frac{1}{p_i} \left|\left\langle A_i x, z_i\right\rangle\right| \\
&\leq \frac{\|A_i\|}{p_i}\|x\|\|z_i\| \\
&= \frac{\|A_i\|}{p_i} * \left(\tau^{k}\sigma_i^kp_i\right)^{1/2}\|x\|_{(\tau^{k})^{-1}}\|z\|_{(P\Sigma^{k})^{-1}}\\
&\leq \sqrt{\beta} \|x\|_{(\tau^{k})^{-1}}\|z\|_{(P\Sigma^{k})^{-1}},
\end{align*}
where the last inequality is a consequence of \eqref{eq:A-ConvergenceCondition}.
}

\end{proof}

\begin{lemma}[A-SPDHG is $\mathcal{C}$-stable]
Under the assumptions of Theorem \ref{thm:A-SPDHG},
\TrackChange{
\begin{itemize}
    \item [(i)] the sequence $(x^k,y^k)_{k \in \mathbb{N}}$ of Algorithm \ref{alg:A-SPDHG} is stable with respect to $\mathcal{C}$ relative to $(N^k)_{k \in \mathbb{N}}$,
    \item [(ii)] the following results hold:
    \begin{align*}
\mathbb{E}\left[  \sum_{k=1}^\infty  \left\|(x^{k+1}-x^k,y^{k}-y^{k-1})\right\|^2 \right] < \infty \quad \text{ and a.s.}\quad  \left\|x^{k+1}-x^k \right\| \rightarrow 0.
\end{align*}
\end{itemize}
}
\label{lm:A-SPDHG_C-stable}
\end{lemma}

\begin{proof}
\TrackChange{Let us begin with the proof of point (i).} By definition of A-SPDHG with serial sampling, the difference between two consecutive dual iterates is almost surely sparse:
\begin{align*}
\TrackChange{\text{a.s. } \forall\, k \in \mathbb{N}\setminus\left\{0\right\}, y^{k}-y^{k-1} \in Y_{\text{sparse}}.}
\end{align*}
Let us define the sequences
\begin{align*}
a^k&= V^{k}(x^{k}-x^*,y^{k}-y^{k-1}) + U^{k}(y^{k}-y^*) ,\quad b^k=V^{k+1}(x^{k+1}-x^k,y^{k}-y^{k-1}),
\end{align*}
which are a.s.\@ non-negative thanks to \eqref{eq:FIe} and \eqref{eq:FIc}. \TrackChange{Notice that the primal iterates $x^l$ from $l=0$ up to $l=k+1$ are measurable with respect to $\mathcal{F}^k$, whereas the dual iterates $y^l$ from $l=0$ up to $l=k$ are measurable with respect to $\mathcal{F}^k$. Hence $a^k$ and $b^k$ are measurable with respect to $\mathcal{F}^k$.} Furthermore, inequalities \eqref{eq:CornerStone}, \eqref{eq:FIa} and \eqref{eq:FIb} imply that almost surely for all $k \in \mathbb{N}\setminus\left\{0\right\}$,
\begin{align*}
\mathbb{E}\left[ a^{k+1} | \mathcal{F}^k \right] \leq (1+\eta^k) a^k - b^k.
\end{align*}
By Robbins-Siegmund lemma \cite{RobbinsSiegmund}, $(a^k)$ converges almost surely, $\sup_k \mathbb{E}\left[ a^k \right] < \infty$ and $\sum_{k=1}^\infty \mathbb{E}\left[ b^k \right] < \infty$. From the last point in particular, we can write thanks to \eqref{eq:FIc} and the monotone convergence theorem:
\begin{align*}
 \mathbb{E}\left[ \sum_{k=1}^\infty \left\|y^{k}-y^{k-1}\right\|_{(P\Sigma^{k+1})^{-1}}^2 \right] &\leq  \mathbb{E}\left[  \sum_{k=1}^\infty  \left\|(x^{k+1}-x^k,y^{k}-y^{k-1})\right\|_{N^{k+1}}^2 \right]\\
 &   \leq \TrackChange{(1-\beta)^{-1}\,}  \mathbb{E}\left[ \sum_{k=1}^\infty  b^k\right] = \TrackChange{(1-\beta)^{-1}}\sum_{k=1}^\infty \mathbb{E}\left[ b^k\right]< \infty,
\end{align*}
hence $ \sum_{k=1}^\infty\|y^{k}-y^{k-1}\|_{(P\Sigma^{k+1})^{-1}}^2$ is almost surely finite, thus $\left(\|y^{k}-y^{k-1}\|_{(P\Sigma^{k+1})^{-1}}^2\right)\TrackChange{_{k\in\mathbb{N}\setminus\left\{0\right\}}}$, \TrackChange{and in turn $\left(\|y^{k}-y^{k-1}\|_{(P\Sigma^{k+1})^{-1}}\right)\TrackChange{_{k\in\mathbb{N}\setminus\left\{0\right\}}}$,} converge almost surely to $0$. Furthermore, $\sup_k \mathbb{E}\left[ a^k \right] < \infty$ hence $\sup_k \|x^{k}-x^*\|_{(\tau^{k})^{-1}}^2$, \TrackChange{and in turn $\sup_k \|x^{k}-x^*\|_{(\tau^{k})^{-1}}$, are finite}, and by \eqref{eq:FIf}, one can write that for $k\in \mathbb{N}\setminus\left\{0\right\}$,
\TrackChange{
\begin{align*}
\left|\left\langle P^{-1}A (x^{k}-x^*), y^{k}-y^{k-1}\right\rangle\right|  &\leq \sqrt{\beta} \|x^{k}-x^*\|_{(\tau^{k+1})^{-1}}\|y^{k}-y^{k-1}\|_{(P\Sigma^{k+1})^{-1}}\\
& \leq \sqrt{\beta(1+\eta^k)}\|x^{k}-x^*\|_{(\tau^{k})^{-1}}\|y^{k}-y^{k-1}\|_{(P\Sigma^{k+1})^{-1}}.
\end{align*}
}
We know that $(\eta^k)^{k\in \mathbb{N}}$ is summable hence converges to $0$. As a consequence, 
\begin{align*}
|\langle P^{-1}A (x^{k}-x^*), y^{k}-y^{k-1}\rangle|  \rightarrow 0 \quad \text{almost surely}.
\end{align*}
To conclude with, thanks to the identity 
\begin{align*}
a^k &= \|(x^{k}-x^*,y^{k}-y^*)\|_{N^k}^2 + \langle P^{-1}A (x^{k}-x^*), y^{k}-y^{k-1} \rangle,\quad k \in \mathbb{N}\setminus\left\{0\right\},
\end{align*}
the almost sure convergence of $(a^k)_{k \in \mathbb{N}}$ implies in turn that of $(\|(x^{k}-x^*,y^{k}-y^*)\|_{N^k}^2)_{k \in \mathbb{N}}$.

\TrackChange{Let us now turn to point (ii).} The first assertion is a straightforward consequence of
\begin{align*}
\mathbb{E}\left[ \sum_{k=1}^\infty  b^k \right] = \sum_{k=1}^\infty \mathbb{E}\left[ b^k \right] < \infty
\end{align*}
and bounds \eqref{eq:FIe} and \eqref{eq:FIc}. Furthermore, it implies that $\sum_{k=1}^\infty\left\|(x^{k+1}-x^k,y^{k}-y^{k-1})\right\|^2$ is a.s.\@ finite, hence $\left(\left\|(x^{k+1}-x^k,y^{k}-y^{k-1})\right\|\right)$ a.s.\@ converges to $0$, and so does $\left(\left\|x^{k+1}-x^k\right\|\right)$.
\end{proof}

\subsection{Weak cluster points of A-SPDHG are saddle-points}
\label{subsec:ClusterPointsAreSaddlePoints}

The goal of this section is to prove that A-SPDHG satisfies point \TrackChange{(iii)} of Theorem \ref{thm:CStableSequencesConvergence}. 
On the event $\left\{I^k=i \right\}$, A-SPDHG update procedure can be rewritten as
\begin{align*}
y_i^{k+1} &= \text{prox}_{\sigma_i^{k+1} f_i^*}(y_i^k + \sigma_i^{k+1} A_i x^{k+1}),\quad \bar{y}_i^{k+1} = y_i^{k+1} + \frac{1}{p_i}\left(y_i^{k+1} - y_i^{k}\right), \quad \bar y_j^{k+1} = y_j^k, j\neq i \\
x^{k+2}& = \text{prox}_{\tau^{k+2} g}(x^{k+1} - \tau^{k+2} A^* \bar{y}^{k+1}).
\end{align*}
\AC{We define  $T_i^{\sigma,\tau}: (x,y)\mapsto (\hat x,\hat y_i)$ by:
\begin{equation*}
\hat{y}_i =  \text{prox}_{\sigma_i f_i^*}(y_i + \sigma_i A_i x),
\quad 
\hat{x} =  \text{prox}_{\tau g}\left( x- \tau A^* y - \tau\frac{1+p_i}{p_i} A^*_i (  \hat{y}_i- y_i ) \right),
\end{equation*}
so that $(x^{k+2},y_i^{k+1}) = T_i^{\sigma_i^{k+1},\tau^{k+2}}(x^{k+1},y^k)$ on the
event $\{I^{k}=i\}$ (and $y^{k+1}_j=y^k_j$ for $j\neq i$).
}

\begin{lemma}[Cluster points of A-SPDHG are saddle points]
\TrackChange{Let $(\bar{x},\bar{y})$ a.s.\@ be a weak cluster point of $(x^k,y^k)_{k \in \mathbb{N}}$ (meaning that there exists a measurable subset $\bar{\Omega}$ of $\Omega$ of probability one such that for all $\omega \in \bar{\Omega}$, $(\bar{x}(\omega),\bar{y}(\omega))$ is a weak sequential
cluster point of $(x^k(\omega),y^k(\omega))_{k \in \mathbb{N}}$)} and assume that the assumptions of Theorem \ref{thm:A-SPDHG} hold. Then $(\bar{x},\bar{y})$ is a.s.\@ in $\mathcal{C}$.
\label{lm:ClusterPointsAreSaddlePoints}
\end{lemma}

\begin{proof}
\TrackChange{Thanks to Lemma \TrackChange{\ref{lm:A-SPDHG_C-stable}-(ii)} and the monotone convergence theorem,
\begin{align*}
\sum_{k=1}^\infty \mathbb{E}\left[ \left\|(x^{k+1}-x^k,y^{k}-y^{k-1})\right\|^2 \right] &=\mathbb{E}\left[  \sum_{k=1}^\infty  \left\|(x^{k+1}-x^k,y^{k}-y^{k-1})\right\|^2 \right] < \infty.
\end{align*}}
\AC{Now, 
\begin{align*}
\sum_{k=1}^\infty \mathbb{E}\left[ \left\|(x^{k+1}-x^k,y^{k}-y^{k-1})\right\|^2 \right] 
&= \sum_{k=1}^\infty \mathbb{E}\left[ \mathbb{E}\left[\|(x^{k+1}-x^k,y^k-y^{k-1})\|^2 | \TrackChange{I^{k-1}}\right] \right] \\
&= \sum_{k=1}^\infty \sum_{i=1}^n \mathbb{P}(I^{k-1}=i)\mathbb{E}\left[ \left\| T_i^{\sigma_i^k,\tau^{k+1}}(x^{\TrackChange{k}},y_i^{\TrackChange{k-1}}) - (x^{k},y_i^{k-1}) \right\|^2\right]\\
&= \mathbb{E}\left[ \sum_{i=1}^n p_i \sum_{k=1}^\infty \left\| T_i^{\sigma_i^k,\tau^{k+1}}(x^{\TrackChange{k}},y^{\TrackChange{k-1}}) - (x^{k},y_i^{k-1}) \right\|^2\right].
\end{align*}
Hence we can deduce that
\begin{align*}
\mathbb{E}\left[ \sum_{k=1}^\infty \sum_{i=1}^n p_i \left\| T_i^{\sigma_i^k,\tau^{k+1}}(x^{\TrackChange{k}},y^{\TrackChange{k-1}}) - (x^{k},y_i^{k-1}) \right\|^2\right] <\infty.
\end{align*}
It follows that the series in the expectation is a.s.~finite, 
and since $p_i>0$ we deduce that almost surely,
\begin{equation}\label{eq:asymptotreg}
\left\| T_i^{\sigma_i^k,\tau^{k+1}}(x^{\TrackChange{k}},y^{\TrackChange{k-1}}) - (x^{k},y_i^{k-1}) \right\|
\stackrel{k\to\infty}\longrightarrow 0
\end{equation}
for all $i=1,\dots n$.
}
\toberemoved{
Now, thanks to Lemma \ref{lm:A-SPDHG_as_PPA}
\begin{align*}
\sum_{k=1}^\infty \mathbb{E}\left[ \left\|(x^{k+1}-x^k,y^{k}-y^{k-1})\right\|^2 \right] 
&= \sum_{k=1}^\infty \mathbb{E}\left[ \mathbb{E}\left[\|(x^{k+1}-x^k,y^k-y^{k-1})\|^2 | \TrackChange{I^{k-1}}\right] \right] \\
&= \sum_{k=1}^\infty \sum_{i=1}^n \mathbb{P}(I^{k-1}=i)\mathbb{E}\left[ \left\| T_i^{\sigma_i^k,\tau^{k+1}}(x^{k+1},y_i^k) - (x^{k},y_i^{k-1}) \right\|^2\right]\\
&= \sum_{i=1}^n p_i \mathbb{E}\left[ \sum_{k=1}^\infty \left\| T_i^{\sigma_i^k,\tau^{k+1}}(x^{k+1},y_i^k) - (x^{k},y_i^{k-1}) \right\|^2\right].
\end{align*}
Hence we can deduce that
\begin{align*}
\sum_{i=1}^n p_i \mathbb{E}\left[ \sum_{k=1}^\infty \left\| T_i^{\sigma_i^k,\tau^{k+1}}(x^{k+1},y_i^k) - (x^{k},y_i^{k-1}) \right\|^2\right] <\infty.
\end{align*}

Let us fix an index  $i \in \llbracket 1,n \rrbracket$. By assumption, $p_i$ is positive, hence the quantity
\begin{align*}
\mathbb{E}\left[\TrackChange{\sum_{k=1}^{\infty}} \| T_i^{\sigma_i^k,\tau^{k+1}}(x^{k+1},y_i^k) - (x^{k},y_i^{k-1}) \|^2\right] 
\end{align*}
is finite, hence the series inside the expectation is a.s.\@ finite, and in turn the summand converges almost surely to $0$:
\begin{align*}
\left\| T_i^{\sigma_i^k,\tau^{k+1}}(x^{k+1},y_i^k) - (x^{k},y_i^{k-1}) \right\| \rightarrow 0 \quad\text{almost surely}.
\end{align*}
}
\AC{
We consider a sample $(x^k,y^k)$ which is bounded and such that~\eqref{eq:asymptotreg} holds. We let for each $i$, $(\hat x^{i,k+1},\hat y_i^{i,k}) = T_i^{\sigma^k_i,\tau^{k+1}}(x^k,y^{k-1})$,
so that $\| (\hat x^{i,k+1},\hat y_i^{i,k})- (x^{k},y_i^{k-1})\|\to 0$ for
$i=1,\dots,n$.
Then, one has
\begin{align*}
\partial f_i^*(\hat y^{i,k}_i) &
\ni \frac{y_i^{k-1}-\hat y_i^{i,k}}{\sigma^k_i} +  A_i x^k =: A_i x^k + \delta_y^{i,k}  \\
\partial g(\hat x^{i,k+1}) & \ni 
\frac{x^{k}-\hat x^{i,k+1}}{\tau^{k+1}} - A^* y^{k-1}
-\frac{1+p_i}{p_i} A^*_i (\hat y_i^{i,k}- y_i^{k-1} ) =: -A^* y^{k-1} + \delta_x^{i,k}
\end{align*}
where  $\delta_{x,y}^{i,k}\to 0$ as $k\to\infty$.
Given a test point $(x,y)$, one may write for any $k$:
\begin{align*}
    f_i^*(y_i) & \ge f_i^*(\hat y^{i,k}_i) + \langle A_i x^k , y_i-y_i^{k-1}\rangle
    + \langle A_i x^k, y_i^{k-1}-\hat y^{i,k}_i\rangle + \langle\delta_y^{i,k},y_i-\hat y^{i,k}_i\rangle,\quad i=1,\dots, n\\
    g(x) & \ge g(\hat x^{1,k+1}) -\langle A^* y^{k-1}, x-x^k \rangle 
    - \langle A^* y^{k-1}, x^k- \hat x^{1,k+1}\rangle + \langle\delta_x^{i,k},x-\hat x^{1,k+1}\rangle
\end{align*}
and summing all these inequalities, we obtain:
\[
g(x)+\sum_{i=1}^n f_i^*(y_i) \ge
g(\hat x^{1,k+1}) +\sum_{i=1}^n \left(f_i^*(\hat y^{i,k}_i)
+ \langle A_i x^k , y_i\rangle\right) -\langle A^* y^{k-1}, x\rangle 
+ \delta^k
\]
where $\delta^k\to 0$ as $k\to\infty$. We deduce that if $(\bar x,\bar y)$ is
the weak limit of a subsequence $(x^{k_l},y^{k_l-1})$ (as well as, of course,
$(x^{k_l},y^{k_l})$), then:
\[
g(x)+\sum_{i=1}^n f_i^*(y_i) \ge
g(\bar x) +\sum_{i=1}^n \left(f_i^*(\bar y_i)
+ \langle A_i \bar x , y_i\rangle\right) -\langle A^* \bar y, x\rangle .
\]
Since $(x, y)$ is arbitrary, we find that~\eqref{eq:SaddlePointEq} holds for $(\bar x,\bar y)$.
}

\end{proof}

\subsection{Proof of Theorem \ref{thm:A-SPDHG}}
\label{subsec:ThmProof}

Under the assumptions of Theorem \ref{thm:A-SPDHG}, the set $\mathcal{C}$ of saddle-points is closed and non-empty and $X\times Y$ is a separable Hilbert space. By Lemma \ref{lm:VariableMetrics}, the variable metrics sequence  $(N^k)_{k \in \mathbb{N}}$ defined in \eqref{eq:MetricDefinition} satisfies condition (i) of Theorem \ref{thm:CStableSequencesConvergence}. Furthermore, the iterates of Algorithm \ref{alg:A-SPDHG} comply with condition (ii) and (iii) of Theorem \ref{thm:CStableSequencesConvergence} by Lemma \ref{lm:A-SPDHG_C-stable} and Lemma \ref{lm:ClusterPointsAreSaddlePoints} respectively, hence converge almost surely to a point in $\mathcal{C}$.

\section{Algorithmic Design and Practical Implementations}
\label{sec:AlgorithmicDesign}

In this section we present practical instances of our A-SPDHG algorithm, where we specify a step-size adjustment rule which satisfies our assumptions in convergence proof. We extend the adaptive step-size balancing rule for deterministic PDHG, which is proposed by \cite{PDHG_adaptive}, into our stochastic setting, with minibatch approximation to minimize the computational overhead.

\TrackChange{
\subsection{A-SPDHG rule (a) -- Tracking $\&$ balancing the primal-dual progress}
}

\TrackChange{
Let's first briefly introduce the foundation of our first numerical scheme, which is built upon the deterministic adaptive PDHG algorithm proposed by Goldstein et al \cite{PDHG_adaptive}, with the iterates:
\begin{equation*}
    x^{k+1} = \text{prox}_{\tau^{k+1} g}(x^k - \tau^{k+1} A^* y^k), \ \  y^{k+1} = \text{prox}_{\sigma^{k+1} f^*}(y^k + \sigma^{k+1} A (2x^{k+1} - x^k))
\end{equation*}

In this foundational work of Goldstein et al \cite{PDHG_adaptive}, they proposed to evaluate two sequences in order to track and balance the progresses of the primal and dual iterates of deterministic PDHG (denoted here as $v_k^*$ and $d_k^*$):
\begin{equation}\label{vds}
     v_k^* := \|(x^k - x^{k+1})/\tau^{k+1} - A^*(y^k - y^{k+1})\|_1, \ \ d_k^* := \|(y^k - y^{k+1})/\sigma^{k+1} - A(x^k - x^{k+1})\|_1.
\end{equation}
These two sequences measure the lengths of the primal and dual subgradients for the objective $\min_{x\in X} \max_{y \in Y} g(x)+\langle Ax, y \rangle - f^*(y)$, which can be demonstrated by the definition of proximal operators. The primal update of deterministic PDHG can be written as:
\begin{equation}
    x^{k+1} = \arg\min_x \frac{1}{2}\|x - (x^k - \tau^{k+1} A^* y^k)\|_2^2 + \tau^{k+1} g(x).
\end{equation}
The optimality condition of the above objective declares: 
\begin{equation}
    0 \in \partial g(x^{k+1}) + A^*y^k + \frac{1}{\tau^{k+1}}(x^{k+1} - x^k).
\end{equation}
By adding $-A^*y^{k+1}$ on both sides and rearranging the terms, one can derive:
\begin{equation}
    (x^k - x^{k+1})/\tau^{k+1} - A^*(y^k - y^{k+1}) \in \partial g(x^{k+1}) + A^*y^{k+1}
\end{equation}
and similarly for the dual update one can also derive:
\begin{equation}
    (y^k - y^{k+1})/\sigma^{k+1} - A(x^k - x^{k+1}) \in \partial f^*(y^{k+1}) -Ax^{k+1},
\end{equation}
which indicates that the sequences $v_k^*$ and $d_k^*$ given by \eqref{vds} should effectively track the primal progress and dual progress of deterministic PDHG, hence Goldstein et al \cite{PDHG_adaptive} propose to utilize these as the basis of balancing the primal and dual step sizes for PDHG.
}

In light of this, we propose our first practical implementation of A-SPDHG in Algorithm \ref{alg:A-SPDHG-G} as our rule-(a), \TrackChange{where we use a unique dual step-size $\sigma^k=\sigma_j^k$ for all iterates $k$ and indices $j$ and where we estimate the progress of achieving optimality on the primal and dual variables via the two sequences $v^k$ and $d^k$ defined at each iteration $k$ with $I^k=i$ as:
\begin{equation}
    v_{k+1} := \|(x^k - x^{k+1})/\tau^{k+1} - \frac{1}{p_i}A_i^*(y_i^k - y_i^{k+1})\|_1,\ \ 
    d_{k+1} := \frac{1}{p_i}\|(y_i^k - y_i^{k+1})/\sigma^{k+1} - A_i(x^k - x^{k+1})\|_1,
\end{equation}
which are minibatch extension of \eqref{vds} tailored for our stochastic setting. By making them balanced on the fly via adjusting the primal-dual step size ratio when appropriate, we can enforce the algorithm to achieve similar progress in both primal and dual steps, hence improve the convergence. To be more specific, as shown in Algorithm \ref{alg:A-SPDHG-G}, in each iteration the values of $v_k$ and $d_k$ are evaluated and compared. If the value of $v_k$ (which tracks the primal subgradients) is significantly larger than $d_k$ (which tracks the dual subgradients), then we know that the primal progress is slower than the dual progress, hence the algorithm would boost the primal step size while shrinking the dual step-size. If $v_k$ is noticeably smaller than $d_k$ then the algorithm would do the opposite.

} Note that here we adopt the choice of $\ell_1$-norm as the length measure for $v^k$ and $d^k$ as done by Goldstein et al \cite{PDHG_adaptive,goldstein2013adaptive}, since we also observe numerically the benefit over the more intuitive choice of $\ell_2$-norm.

For full-batch case ($n=1$), it reduces to the adaptive PDHG proposed by \cite{PDHG_adaptive,goldstein2013adaptive}. We adjust the ratio between primal and dual step sizes according to the ratio between $v^k$ and $d^k$, and whenever the step-sizes change, we shrink $\alpha$ (which controls the amplitude of the changes) by a factor $\eta \in (0,1)$ -- we typically choose $\eta = 0.995$ in our experiments. For the choice of $s$, we choose $s=\|A\|$ as our default.\footnote{The choice of $s$ is crucial for the convergence behavior of rule (a), and we found numerically that it is better to scale with the operator norm $\|A\|$ instead of depending on the range of pixel values as suggested in \cite{goldstein2013adaptive}.}


\TrackChange{\subsubsection{Reducing the overhead with subsampling:}} Noting that unlike the deterministic case which does not have the need of extra matrix-vector multiplication since $A^*y^k$ and $Ax^k$ can be memorized, our stochastic extension will require the computation of $A_i x^k$ since we will sample different subsets between back-to-back iterations with high probability. \TrackChange{When using this strategy, we will only have a maximum $50\%$ overhead in terms of FLOP counts, which is numerically negligible compared to the significant acceleration it will bring towards SPDHG especially when the primal-dual step-size ratio is suboptimal, as we will demonstrate later in the experiments. Moreover, we found numerically that we can significantly reduce this overhead by approximation tricks such as subsampling:}
\begin{equation}\label{approx_d}
    d^{k+1} \approx \frac{\rho}{p_i}\|S^k(y_i^k - y_i^{k+1})/\sigma^{k+1} - S^kA_i(x^k - x^{k+1})\|_1
\end{equation}
with $S^k$ being a random subsampling operator \TrackChange{such that} $\mathbb E [(S^k)^TS^k] = \frac{1}{\rho}\text{Id}$. In our experiments we choose $10\%$ subsampling for this approximation hence the overhead is reduced from $50\%$ to only $5\%$ which is negligible, \TrackChange{without compromising the convergence rates in practice}.

\begin{algorithm}[t]
\caption{A-SPDHG, rule (a)} 
\label{alg:A-SPDHG-G}
\begin{algorithmic}
\STATE{Input: dual step-size $\sigma^0$, primal step-size $\tau^0$,  $\alpha^0 \in (0,1)$, $\eta \in (0,1)$, $\delta > 1$, probabilities $(p_i)_{1\leq i\leq n}$; primal variable $x^0$, dual variable $y^0$}
\STATE{Initialize $\bar{y}^0=y^0$, $v^0=d^0=0$, $s = \|A\|$}
\FOR{$k\in \llbracket 0, K-1 \rrbracket$}
\STATE{\textbf{If}\ $v^{k} > sd^{k}\delta$\ \textbf{then}\ $\tau^{k+1} =\frac{\tau^k}{1-\alpha^k}$, $\sigma^{k+1} = \sigma^k(1- \alpha^k)$, $\alpha^{k+1} = \alpha^k \eta$}
\STATE{\textbf{If}\ $v^{k} < sd^{k}/\delta$\ \textbf{then}\ $\tau^{k+1} =\tau^k(1-\alpha^k)$, $\sigma^{k+1} = \frac{\sigma^k}{1- \alpha^k}$, $\alpha^{k+1} = \alpha^k \eta$}
\STATE{\textbf{If}\ $sd^{k}/\delta \leq v^{k} \leq sd^{k}\delta$\ \textbf{then}\ $\tau^{k+1} =\tau^k$, $\sigma^{k+1} =\sigma^k$, $\alpha^{k+1} = \alpha^k$}
\STATE{$x^{k+1} = \text{prox}_{\tau^{k+1} g}(x^k - \tau^{k+1} A^* \bar{y}^k)$}
\STATE{Randomly pick $i \in \llbracket 1,n \rrbracket$ with probability $p_i$}
\STATE{$y_j^{k+1} = \begin{cases} 
            \text{prox}_{\sigma^{k+1} f_i^*}(y_i^k + \sigma^{k+1} A_i x^{k+1}) & \text{if } j=i\\
            y_j^k & \text{if } j\neq i
            \end{cases}$}
\STATE{$\bar{y}_j^{k+1} = \begin{cases} 
            y_i^{k+1} + \frac{1}{p_i}\left(y_i^{k+1} - y_i^{k}\right) & \text{if } j=i\\
            y_j^{k} & \text{if } j\neq i
            \end{cases}$}
\STATE{$v^{k+1} = \|(x^k - x^{k+1})/\tau^{k+1} - \frac{1}{p_i}A_i^*(y_i^k - y_i^{k+1})\|_1$}         \STATE{$d^{k+1} = \frac{1}{p_i}\|(y_i^k - y_i^{k+1})/\sigma^{k+1} - A_i(x^k - x^{k+1})\|_1$ -- or approximate this step by \eqref{approx_d}}      
\ENDFOR
\RETURN $x^K$
\end{algorithmic}
\end{algorithm}

\TrackChange{\subsection{A-SPDHG rule (b) -- Exploiting angle alignments}}
More recently, Yokota and Hontani \cite{yokota2017efficient} propose a variant of adaptive step-size balancing scheme for PDHG, utilizing the angles between the subgradients $\partial g(x^{k+1}) + A^*y^{k+1}$ and the difference of the updates $x^k - x^{k+1}$.

If these two directions are highly aligned, then the primal step size can be increased for bigger step. If these two directions have a large angle, then the primal step-size should be shrunken. By extending this scheme to stochastic setting we obtain another choice of adaptive scheme for SPDHG.

We present this scheme in Algorithm \ref{alg:A-SPDHG-AC} as our rule (b). At iteration $k$ with $I^k=i$, compute:
\begin{equation}
    q^{k+1} = (x^k - x^{k+1})/\tau^{k+1} - \frac{1}{p_i}A_i^*(y_i^k - y_i^{k+1}),
\end{equation}
as an estimate of $\partial g(x^{k+1}) + A^*y^{k+1}$, then measure the cosine of the angle between this and $x^k - x^{k+1}$:
\begin{equation}
    w^{k+1} = \frac{\langle x^k -x^{k+1}, q^{k+1} \rangle}{(\|x^k -x^{k+1}\|_2\|q^{k+1}\|_2)}.
\end{equation}
The threshold $c$ for the cosine value (which triggers the increase of the primal step-size) typically needs to be very close to 1 (we use $c=0.999$) \TrackChange{due to the fact that we mostly apply these type of algorithms in high-dimensional problems, following the choice in \cite{yokota2017efficient} which was for deterministic PDHG}.

\begin{algorithm}[t]
\caption{A-SPDHG, rule (b)} 
\label{alg:A-SPDHG-AC}
\begin{algorithmic}
\STATE{Input: dual step-size $\sigma^0$, primal step-size $\tau^0$, $\eta \in (0,1)$, probabilities $(p_i)_{1\leq i\leq n}$; primal variable $x^0$, dual variable $y^0$}
\STATE{Initialize $\bar{y}^0=y^0$, $w^0=0$,  $\alpha^0 = 1$}
\FOR{$k\in \llbracket 0, K-1 \rrbracket$}
\STATE{\textbf{If}\ $w^{k} < 0$\ \textbf{then}\ $\tau^{k+1} =\frac{\tau^k}{1+\alpha^k}$, $\sigma^{k+1} = \sigma^k(1+\alpha^k)$, $\alpha^{k+1} = \alpha^k \eta$}
\STATE{\textbf{If}\  $w^{k} \geq c$\ \textbf{then}\ $\tau^{k+1} =\tau^k(1+\alpha^k)$, $\sigma^{k+1} = \frac{\sigma^k}{1+ \alpha^k}$, $\alpha^{k+1} = \alpha^k \eta$}
\STATE{\textbf{If}\ $0 \leq w^{k} < c$\ \textbf{then}\ $\tau^{k+1} =\tau^k$, $\sigma^{k+1} =\sigma^k$, $\alpha^{k+1} = \alpha^k$}
\STATE{$x^{k+1} = \text{prox}_{\tau^{k+1} g}(x^k - \tau^{k+1} A^* \bar{y}^k)$}
\STATE{Randomly pick $i \in \llbracket 1,n \rrbracket$ with probability $p_i$}
\STATE{$y_j^{k+1} = \begin{cases} 
            \text{prox}_{\sigma^{k+1} f_i^*}(y_i^k + \sigma^{k+1} A_i x^{k+1}) & \text{if } j=i\\
            y_j^k & \text{if } j\neq i
            \end{cases}$}
\STATE{$\bar{y}_j^{k+1} = \begin{cases} 
            y_i^{k+1} + \frac{1}{p_i}\left(y_i^{k+1} - y_i^{k}\right) & \text{if } j=i\\
            y_j^{k} & \text{if } j\neq i
            \end{cases}$}
\STATE{$q^{k+1} = (x^k - x^{k+1})/\tau^{k+1} - \frac{1}{p_i}A_i^*(y_i^k - y_i^{k+1})$} 
\STATE{$w^{k+1} = \langle x^k -x^{k+1}, q^{k+1} \rangle / (\|x^k -x^{k+1}\|_2\|q^{k+1}\|_2)$}
\ENDFOR
\RETURN $x^K$
\end{algorithmic}
\end{algorithm}

Recently Zdun et al \cite{A-SPDHG_MPI} proposed a heuristic similar to our rule (b), but they choose $q^{k+1}$ to be the approximation for an element of $\partial g(x^{k+1})$ instead of $\partial g(x^{k+1}) + A^*y^{k+1}$. Our choice follows more closely to the original scheme of Yokota and Hontani \cite{yokota2017efficient}. We numerically found that their scheme is not competitive in our settings.


\clearpage
\begin{figure}[th]
   \centering

      \subfloat[starting ratio $10^{-3}$ ]{\includegraphics[width= .475\textwidth]{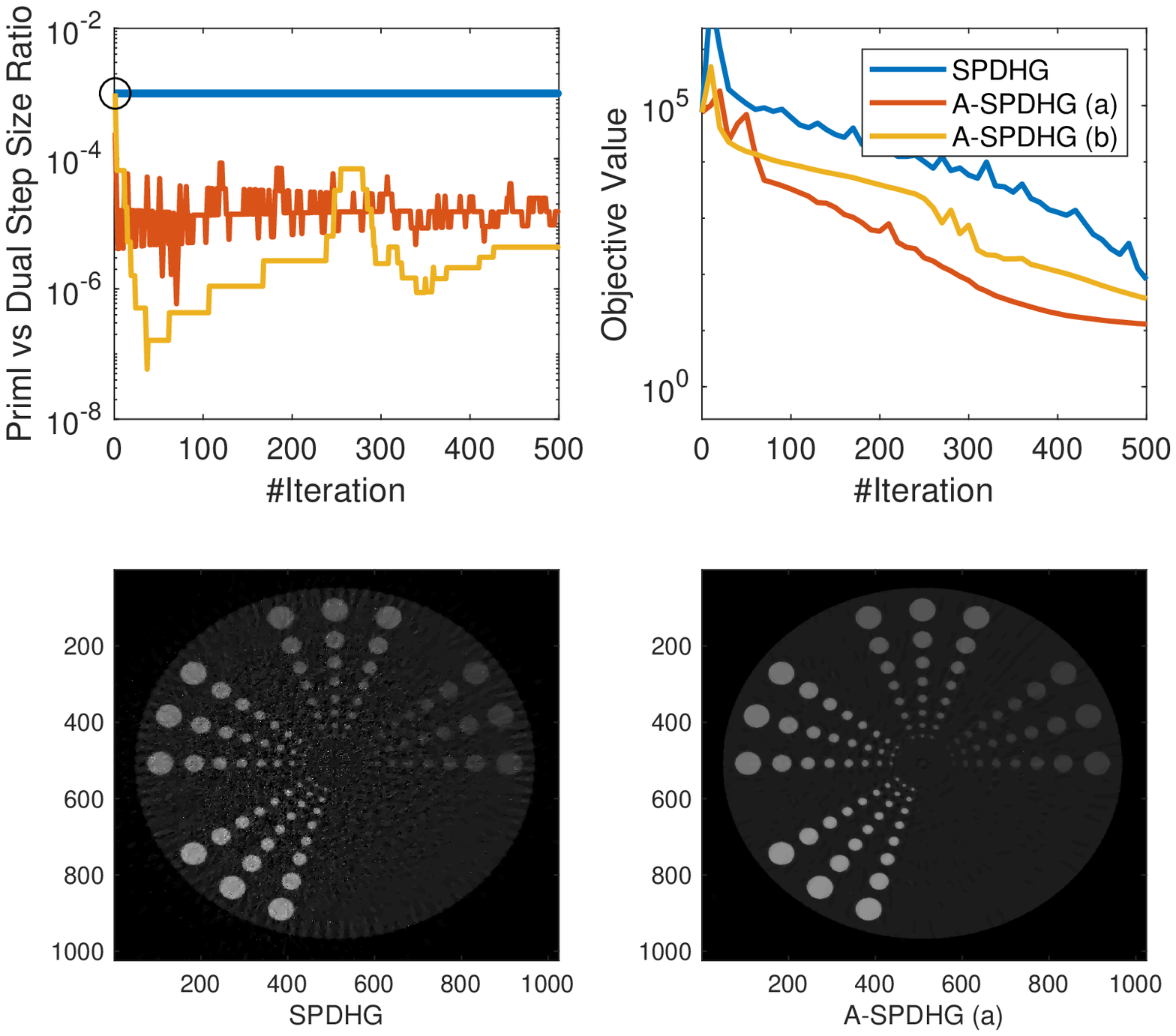}}
    \subfloat[starting ratio $10^{-5}$ ]{\includegraphics[width= .475\textwidth]{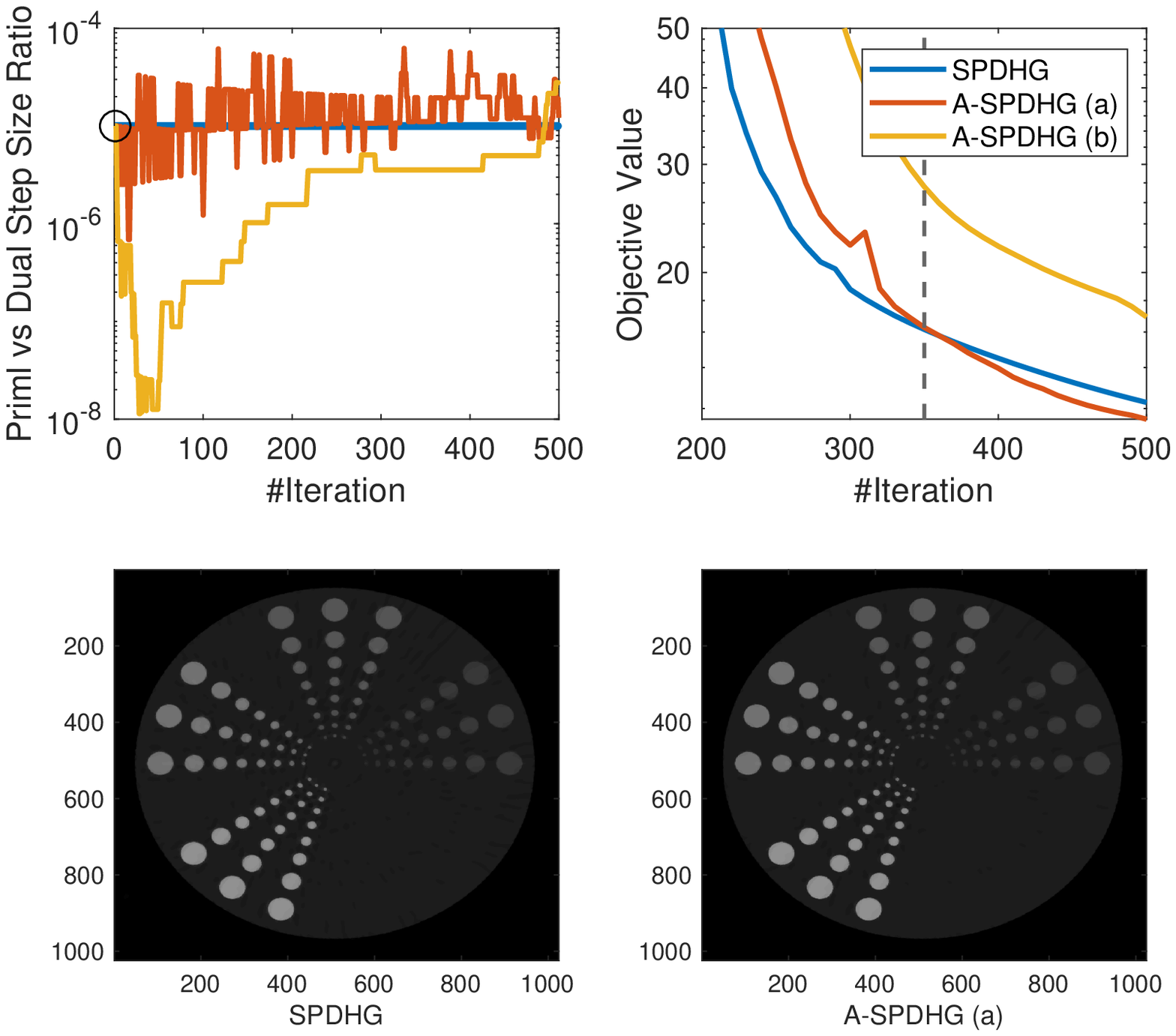}}\\
      \subfloat[starting ratio $10^{-7}$ ]{\includegraphics[width= .475\textwidth]{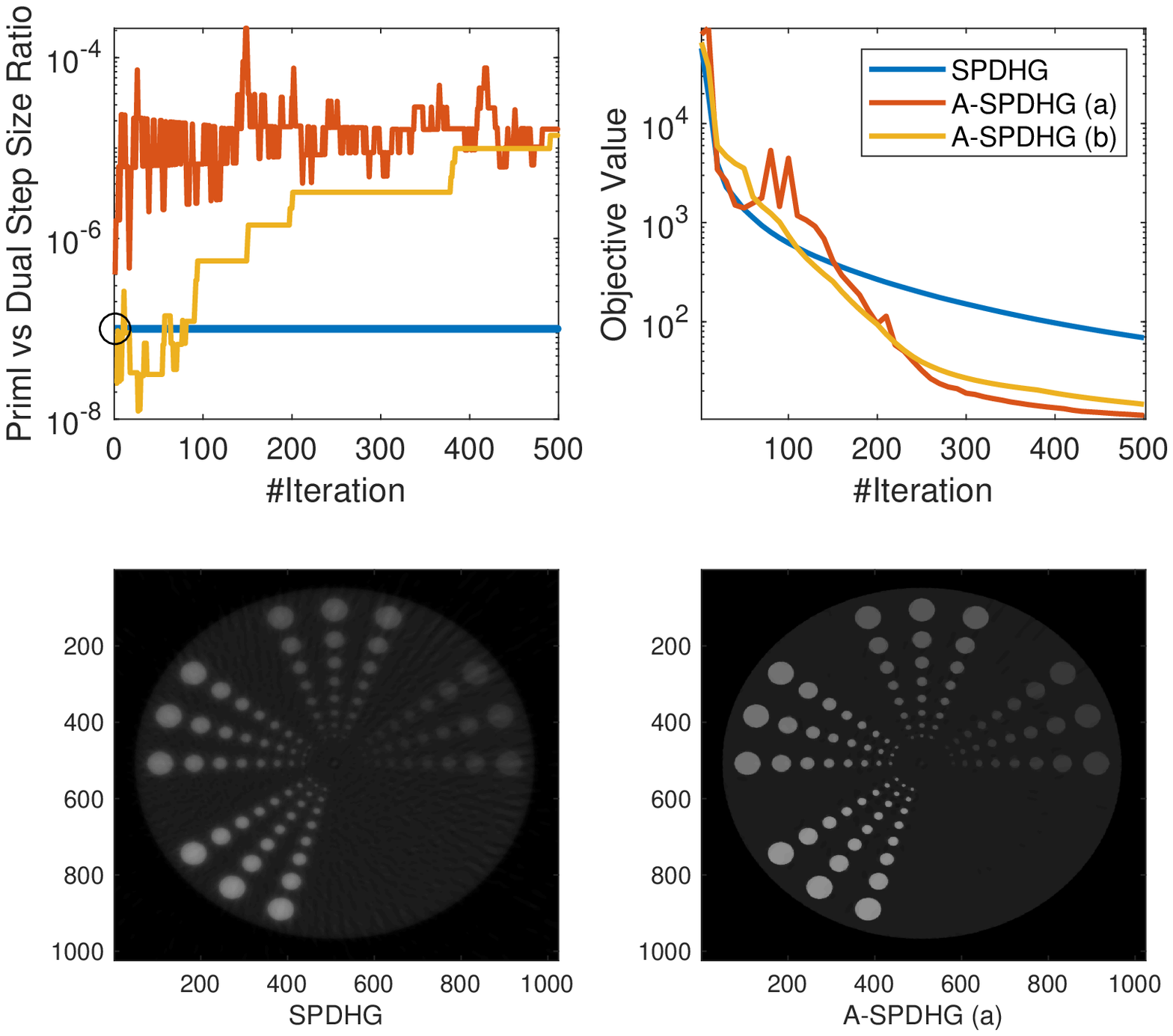}}
          \subfloat[starting ratio $10^{-9}$ ]{\includegraphics[width= .475\textwidth]{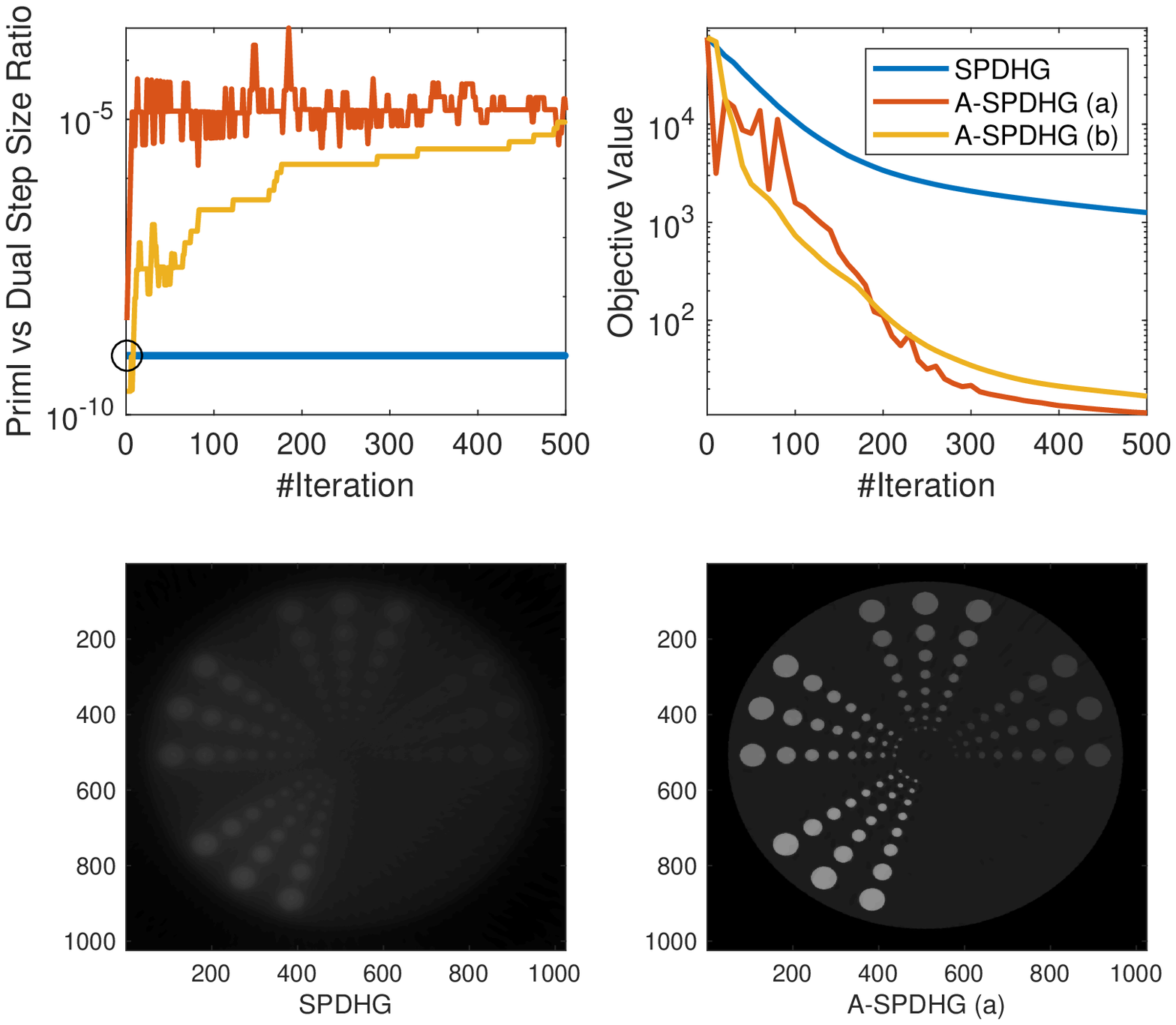}}
   \caption{Comparison between SPDHG and A-SPDHG on Sparse-View CT (Example 1), with a variety of starting primal-dual step-size ratios. Here the forward operator $A\in\mathbb{R}^{m\times d}$ where the dimension $m=368640$, $d= 1048576$. We include the images reconstructed by the algorithms at termination (50th epoch). In the first plot of each subfigure, the black circle indicates the starting step-size ratio for all the algorithms, same for the following figures.}
\label{f1}
\end{figure}

\begin{figure}[t]
   \centering

      \subfloat[starting ratio $10^{-3}$ ]{\includegraphics[width= .475\textwidth]{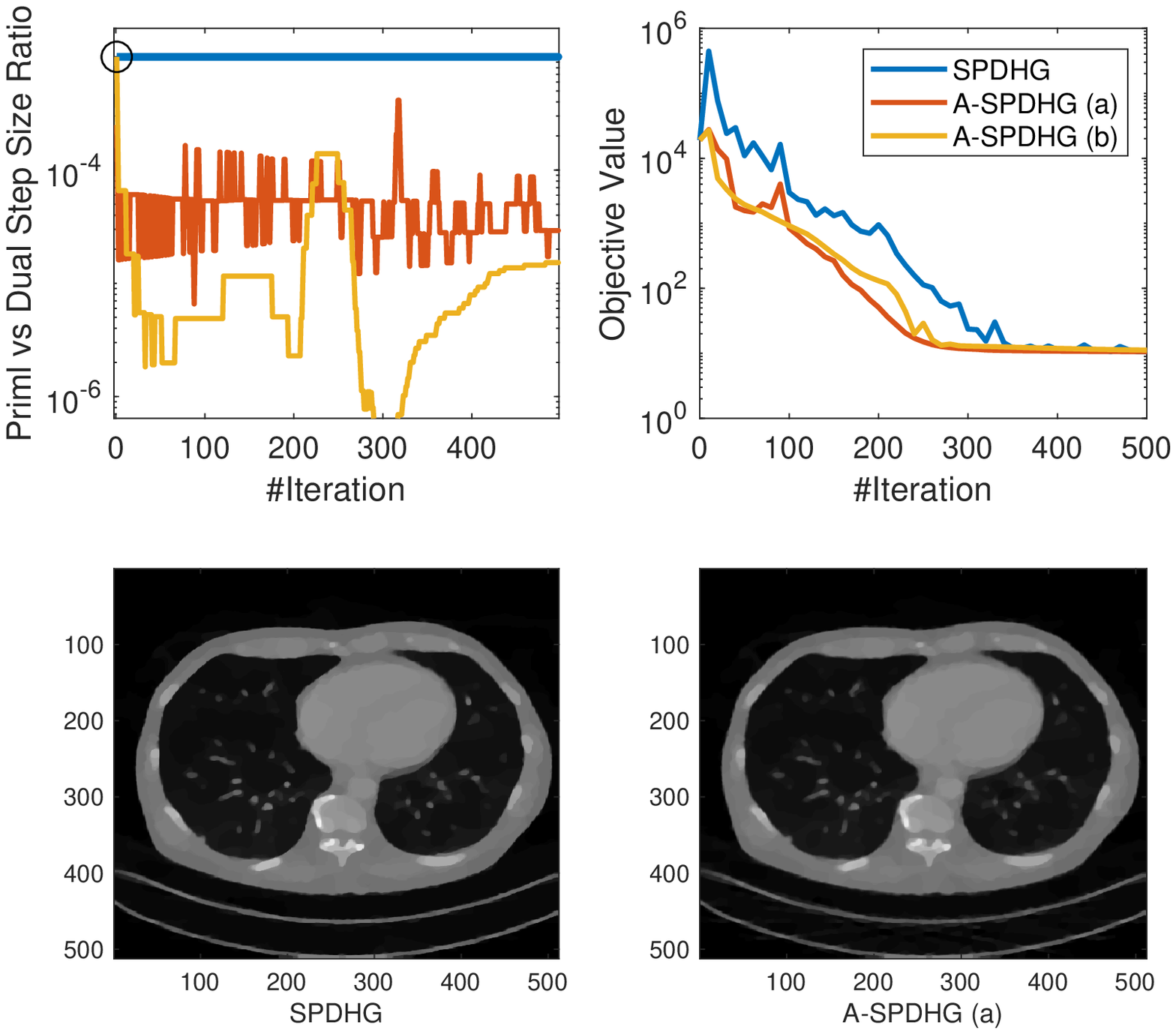}}
    \subfloat[starting ratio $10^{-5}$ ]{\includegraphics[width= .475\textwidth]{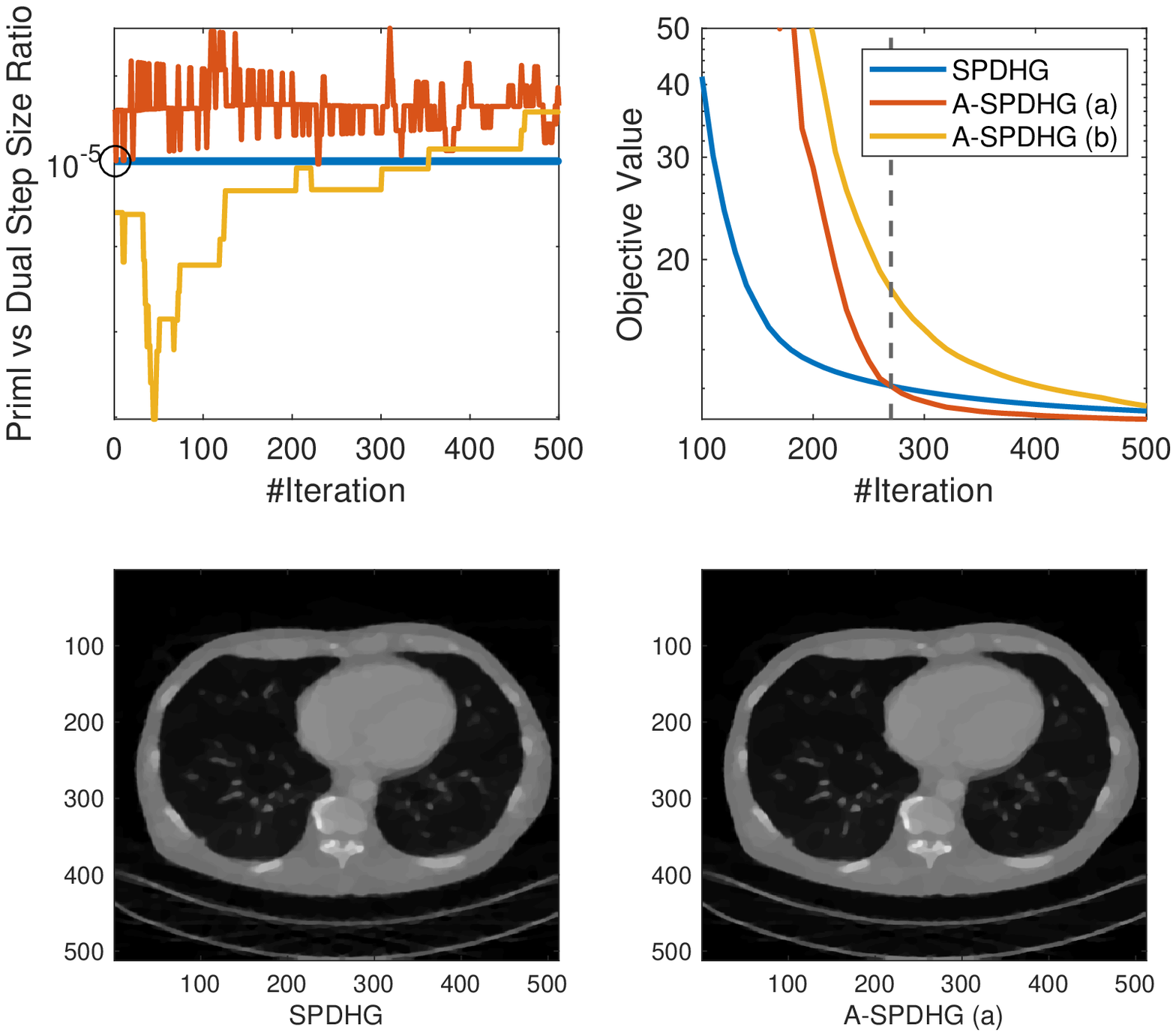}}\\
      \subfloat[starting ratio $10^{-7}$ ]{\includegraphics[width= .475\textwidth]{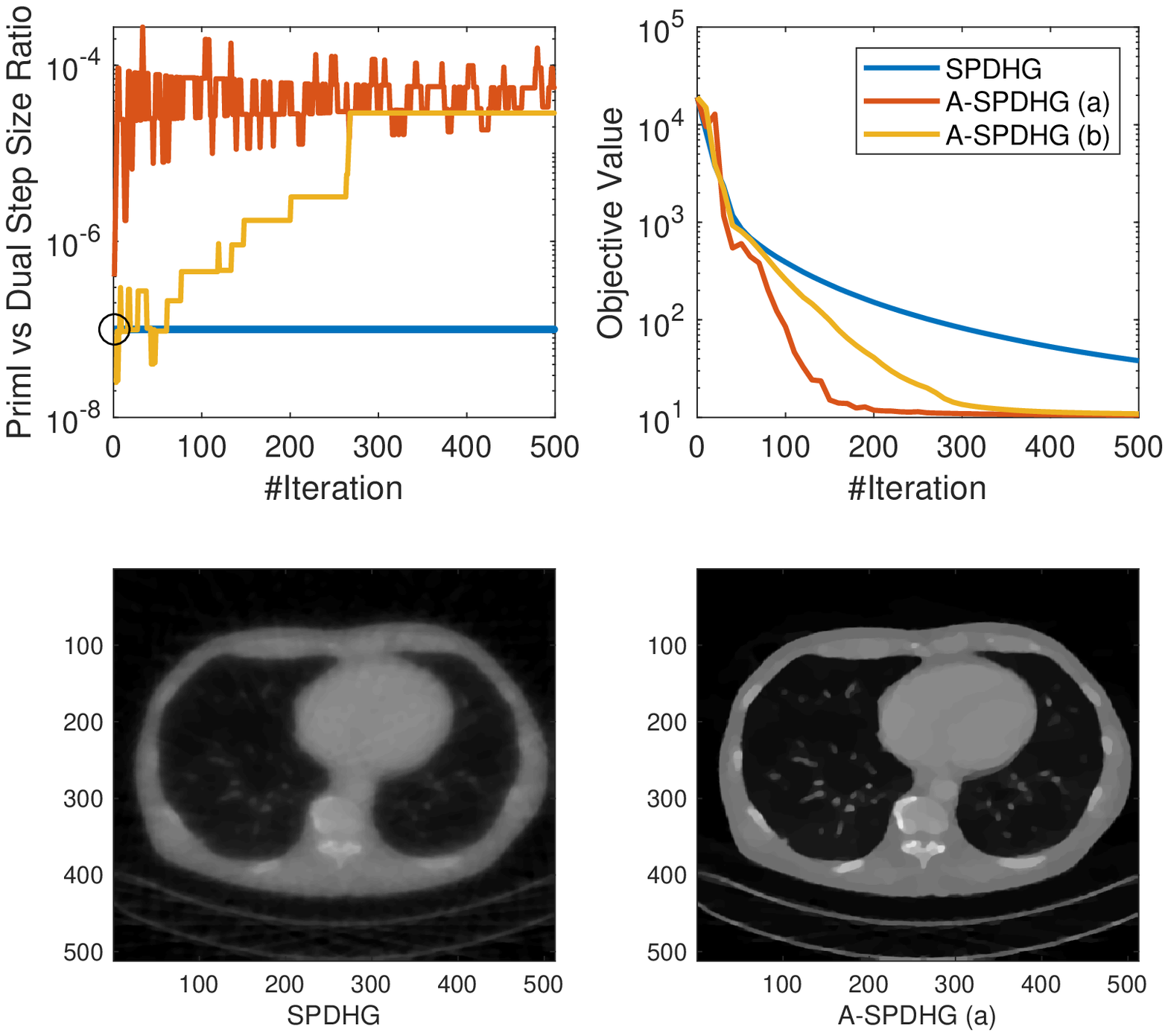}}
          \subfloat[starting ratio $10^{-9}$ ]{\includegraphics[width= .475\textwidth]{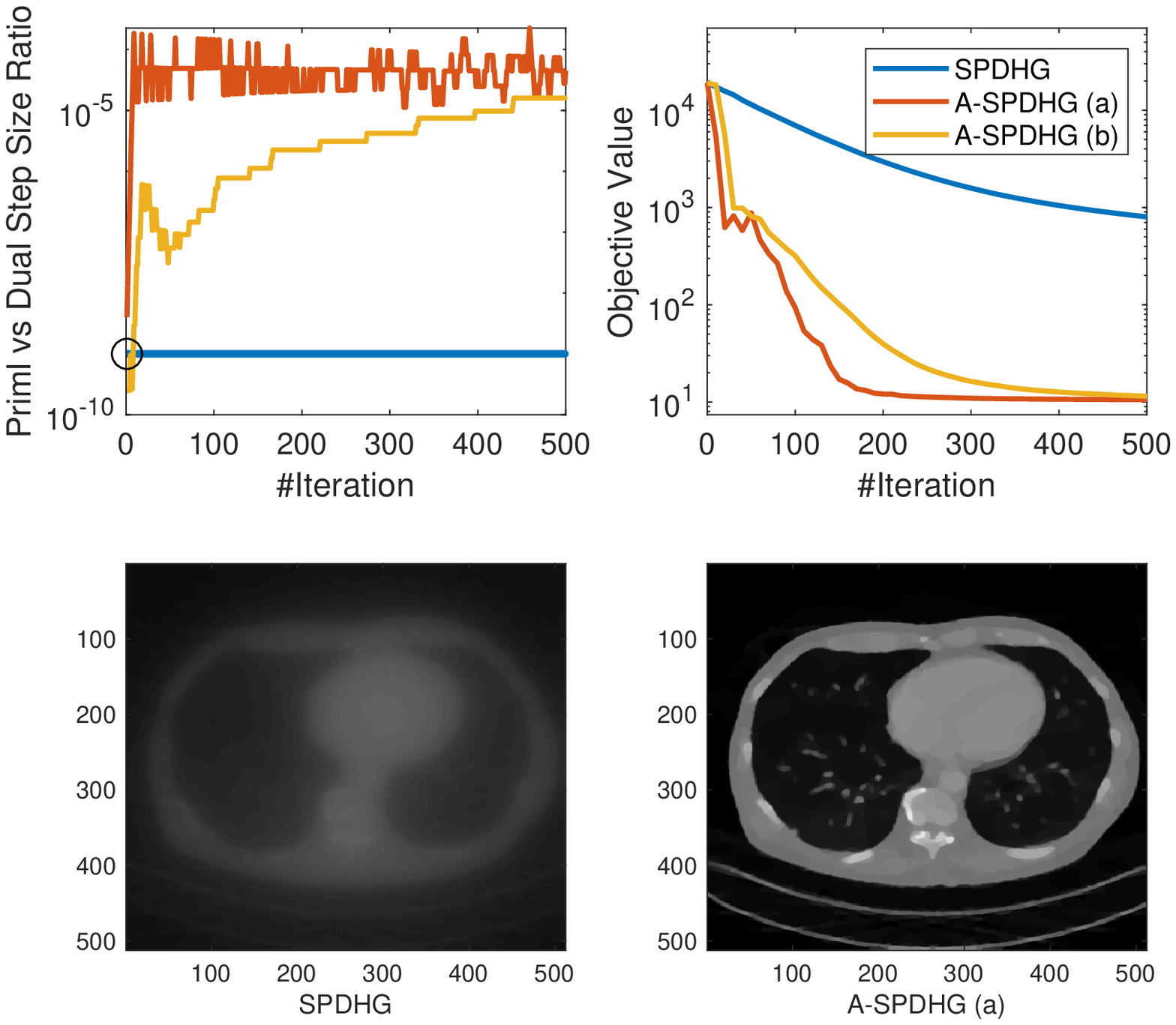}}
   \caption{Comparison between SPDHG and A-SPDHG on Sparse-View CT (Example 2), with a variety of starting primal-dual step-size ratios. Here the forward operator $A\in\mathbb{R}^{m\times d}$ where the dimension $m=92160$, $d= 262144$. We include the images reconstructed by the algorithms at termination (50th epoch).}
\label{f2}
\end{figure}

\begin{figure}[t]
   \centering

      \subfloat[starting ratio $10^{-3}$ ]{\includegraphics[width= .475\textwidth]{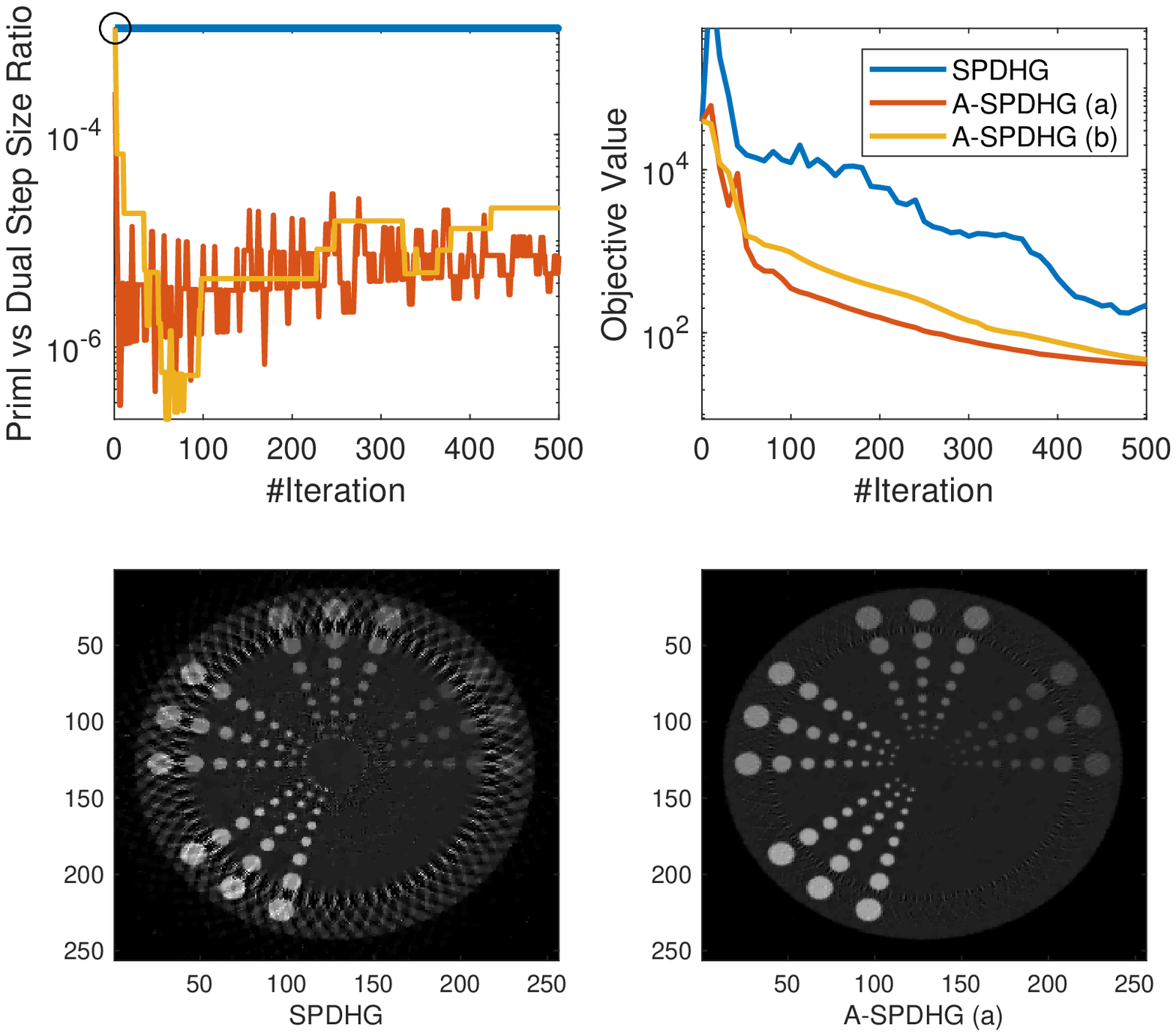}}
    \subfloat[starting ratio $10^{-5}$ ]{\includegraphics[width= .475\textwidth]{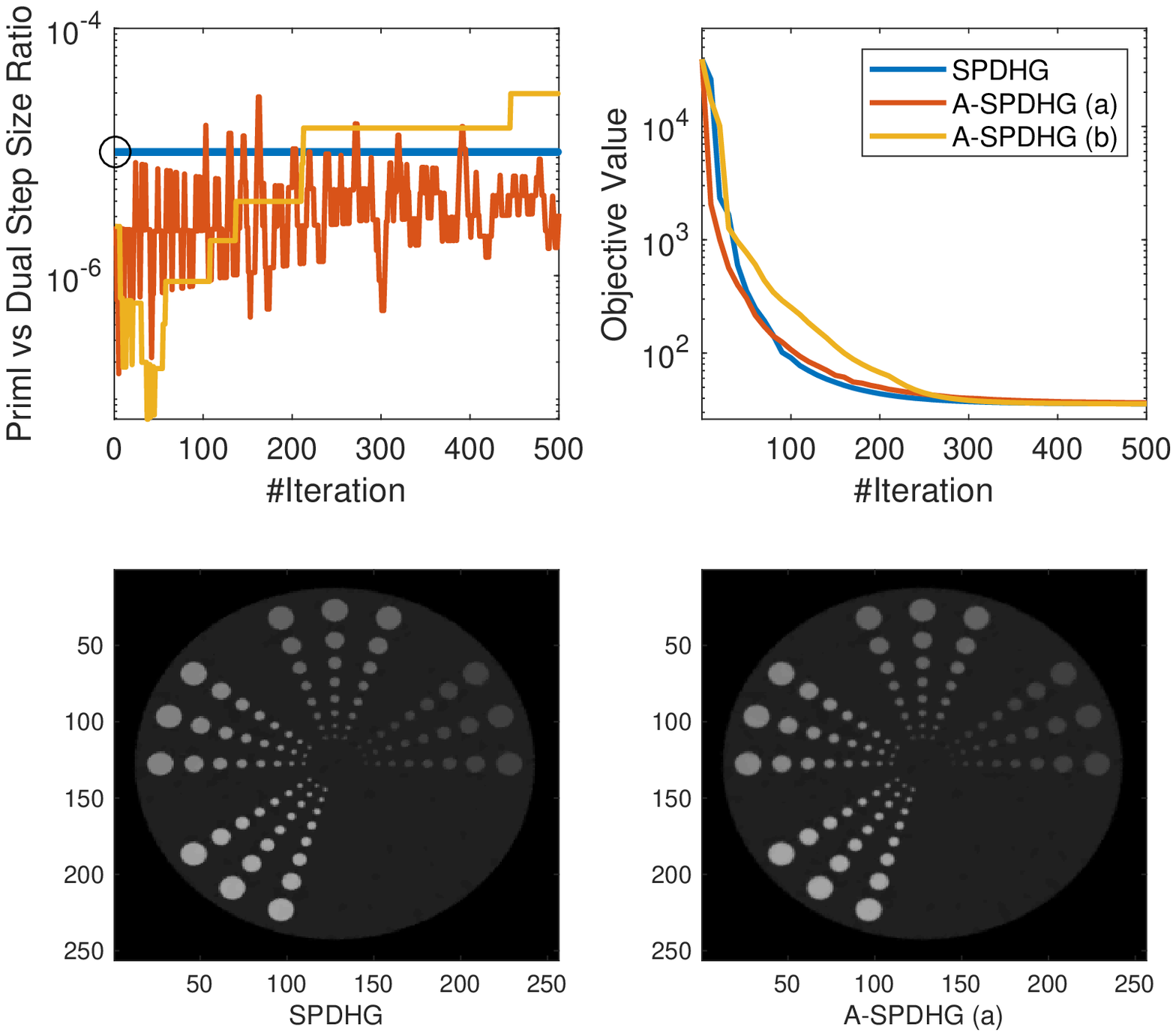}}\\
      \subfloat[starting ratio $10^{-7}$ ]{\includegraphics[width= .475\textwidth]{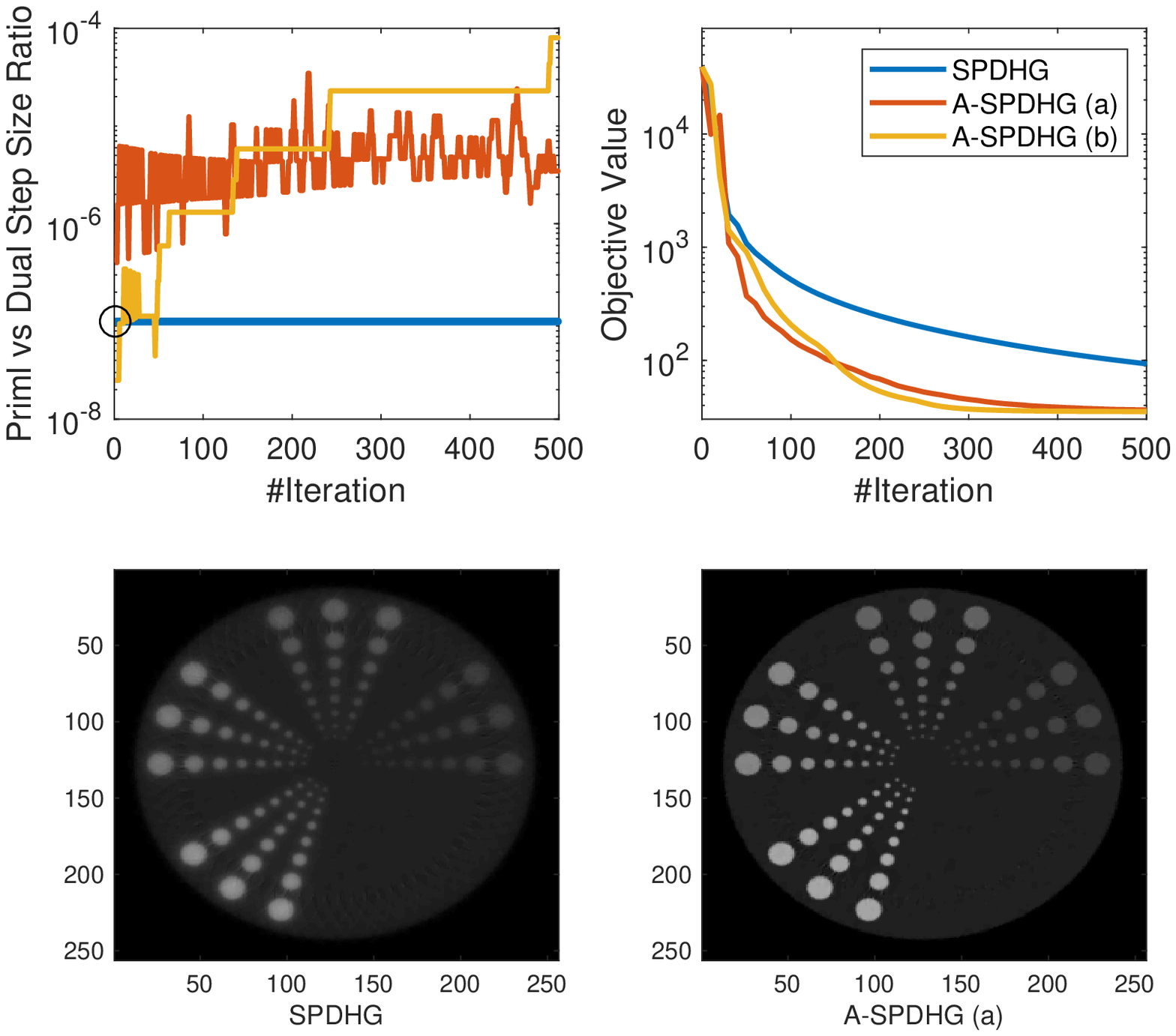}}
          \subfloat[starting ratio $10^{-9}$ ]{\includegraphics[width= .475\textwidth]{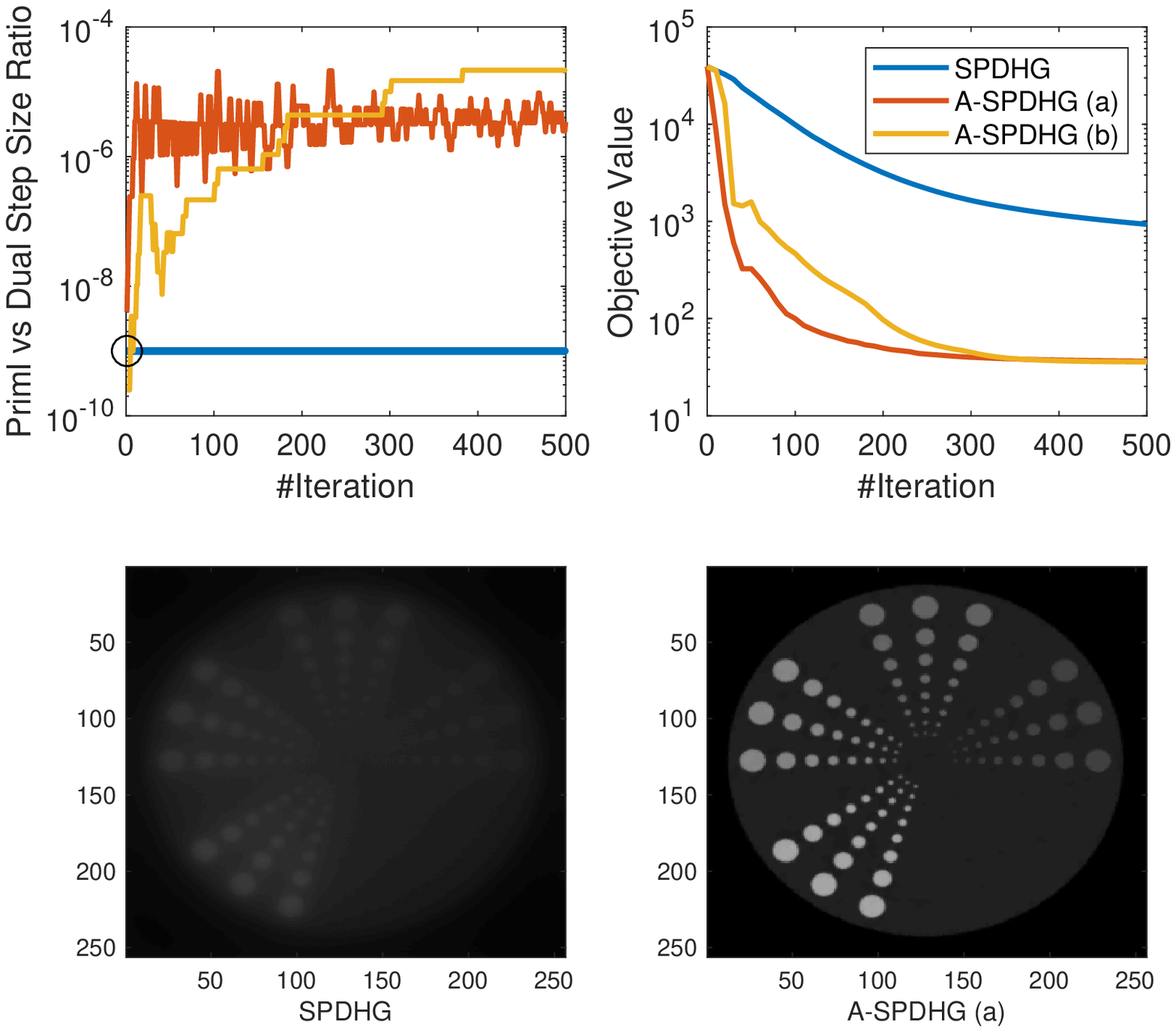}}
   \caption{Comparison between SPDHG and A-SPDHG on Low-Dose CT (where we use a large number of highly-noisy X-ray measurements), with a variety of starting primal-dual step-size ratios. Here the forward operator $A\in\mathbb{R}^{m\times d}$ where the dimension $m=184320$, $d= 65536$. We resized the phantom image to 256 by 256. We include the images reconstructed by the algorithms at termination (50th epoch).}
\label{f2b}
\end{figure}

\begin{figure}[th]
   \centering
     \subfloat[starting ratio $10^{-3}$ ]{\includegraphics[width= .475\textwidth]{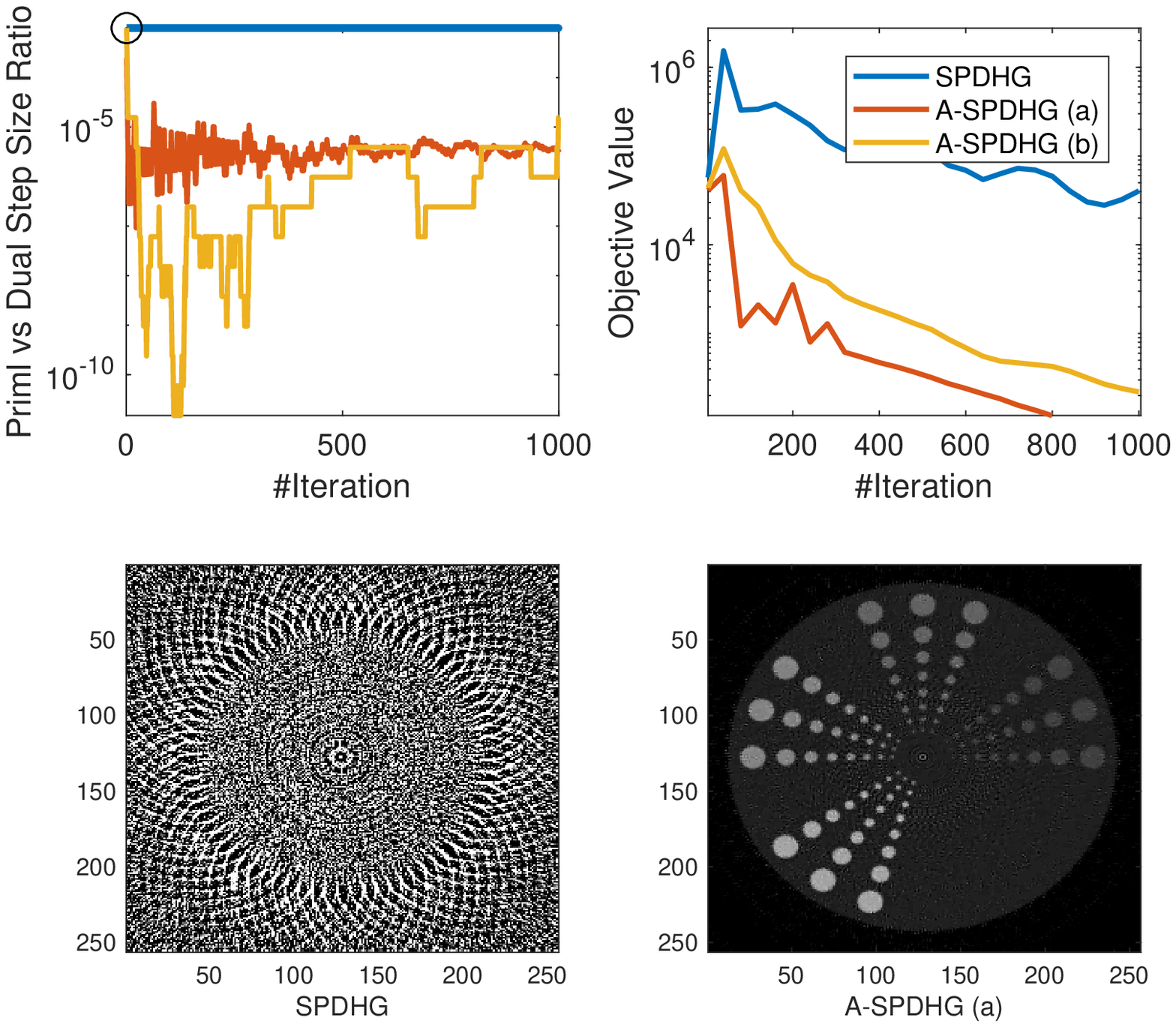}}
    \subfloat[starting ratio $10^{-5}$ ]{\includegraphics[width= .475\textwidth]{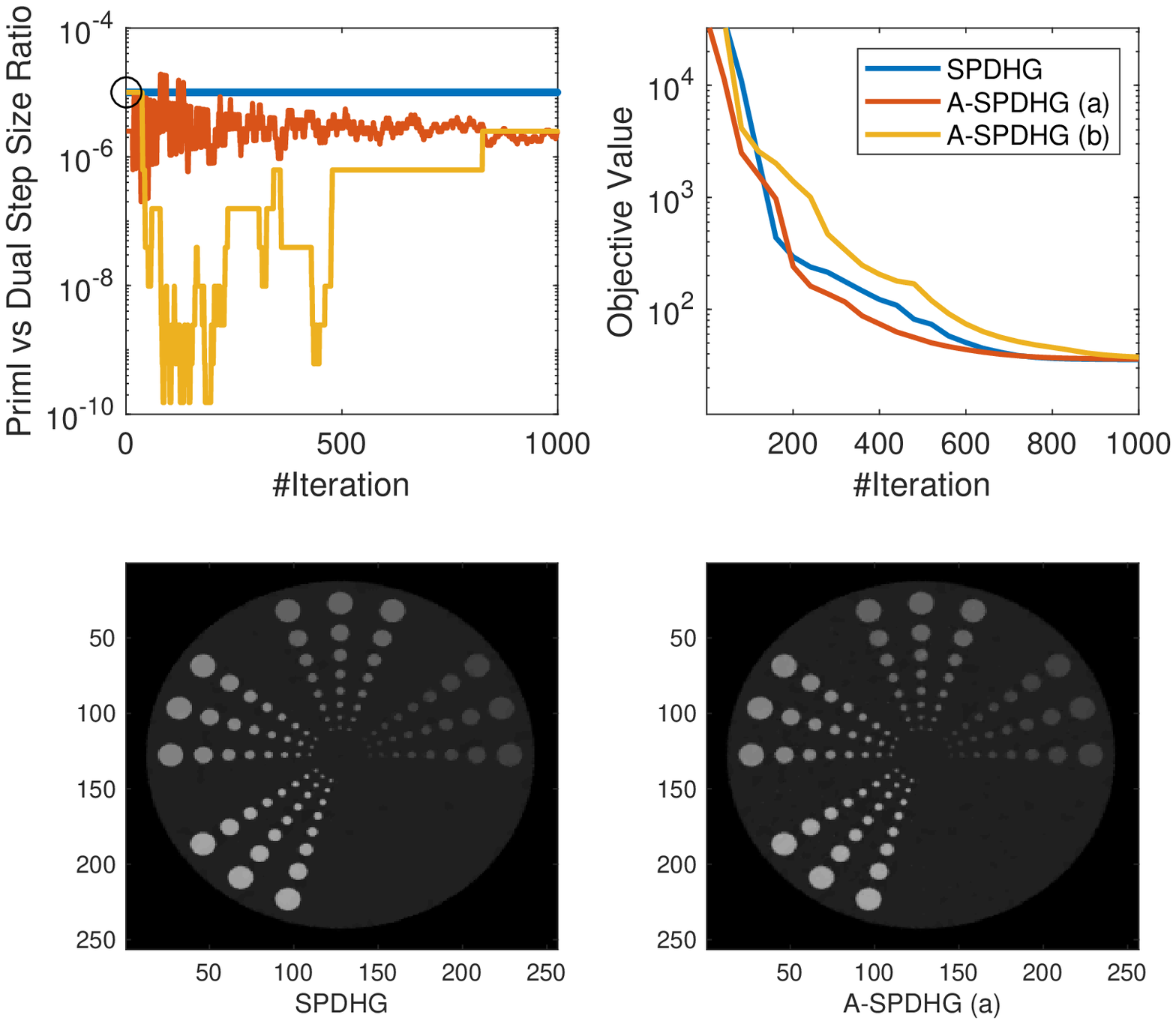}}\\
      \subfloat[starting ratio $10^{-7}$ ]{\includegraphics[width= .475\textwidth]{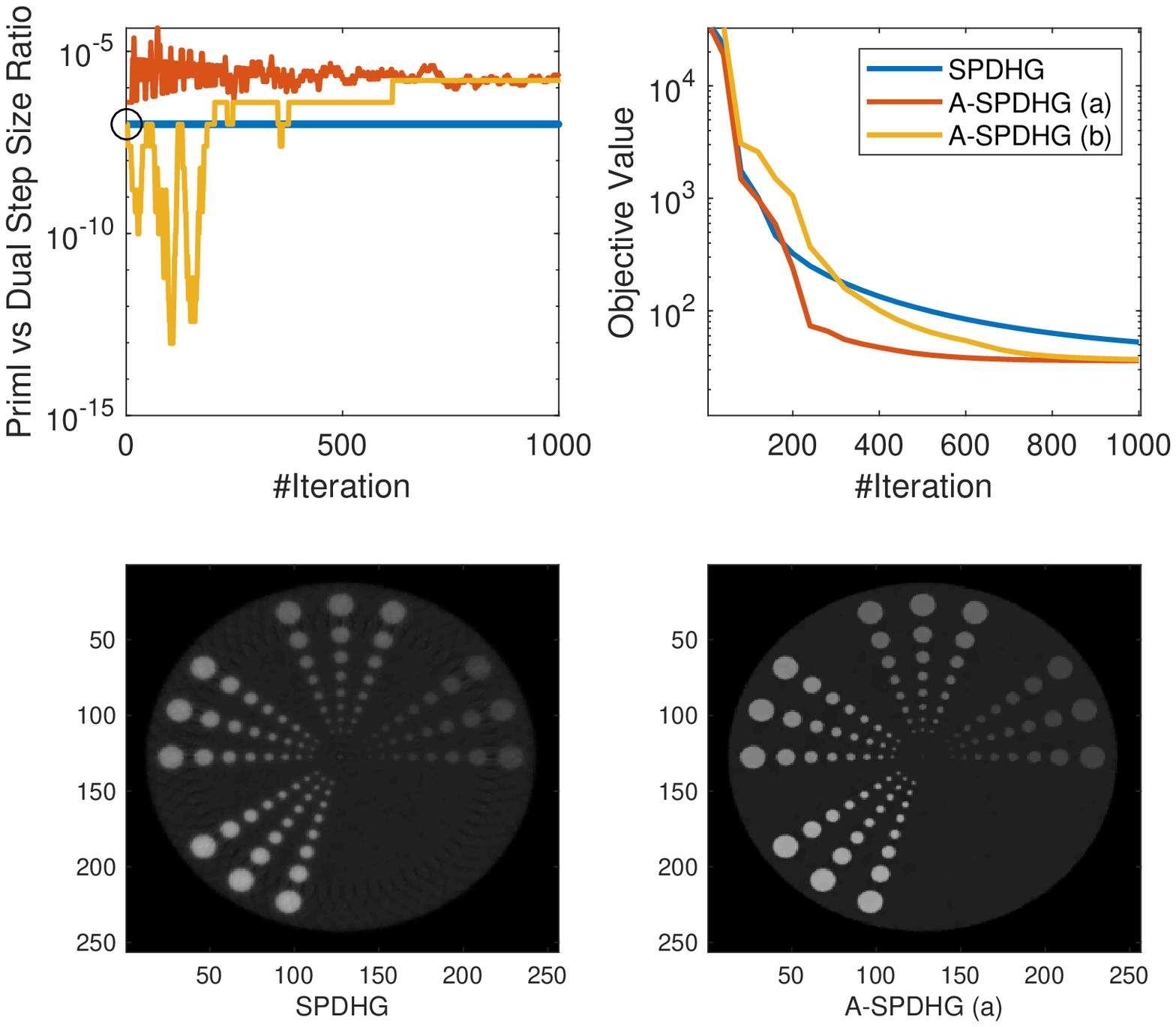}}
      \subfloat[starting ratio $10^{-9}$ ]{\includegraphics[width= .475\textwidth]{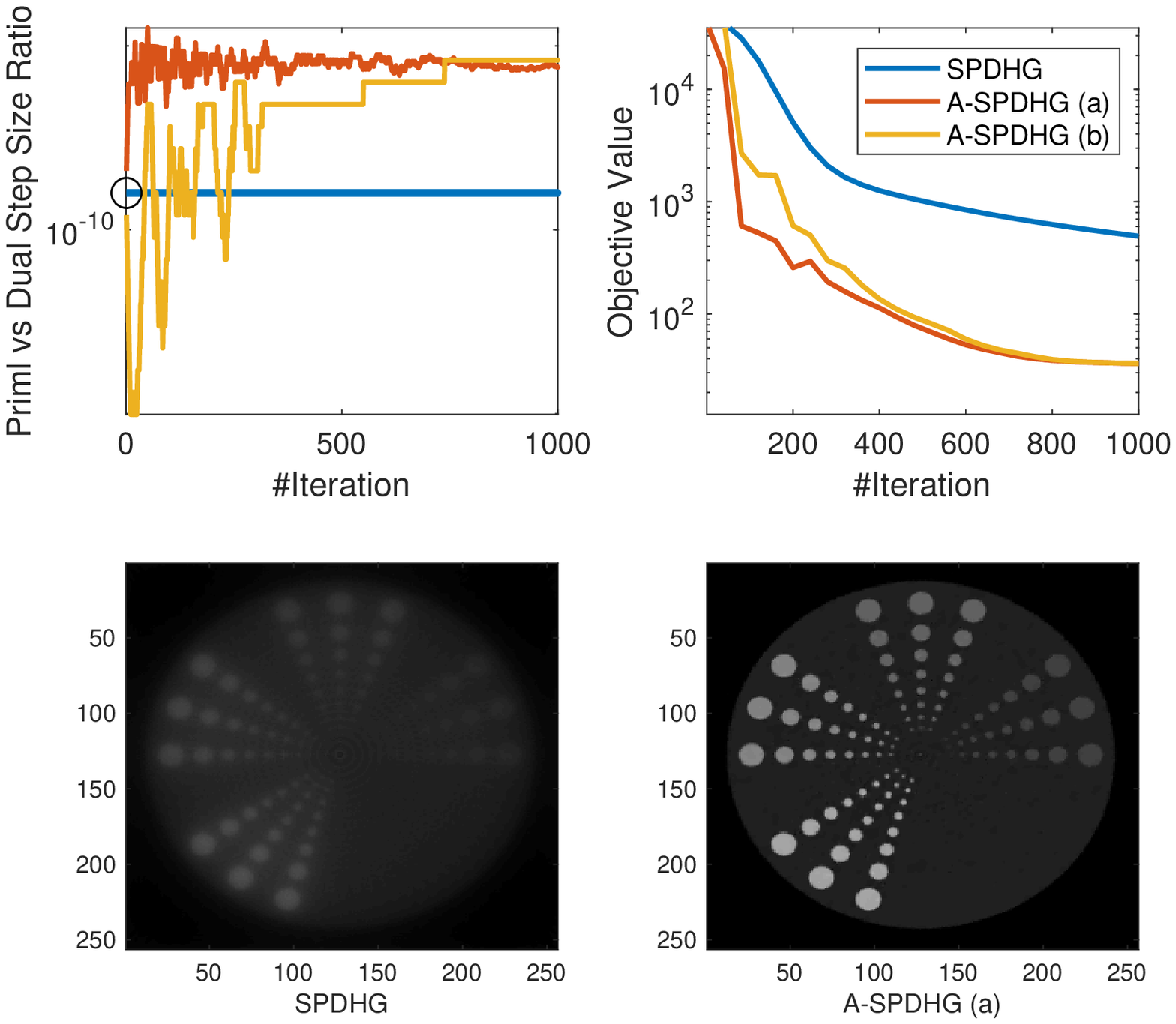}}\\
   \caption{Comparison between SPDHG and A-SPDHG \TrackChange{with the data being splitted to 40 minibatches on Low-Dose CT. Comparing to the results presented in Fig.\ref{f2b} which used 10 minibatches, we obtain similar results and our A-SPDHG continues to perform more favorably comparing to SPDHG.}.}
\label{f2_40}
\end{figure}

\begin{figure}[t]
   \centering

      \subfloat[starting ratio $10^{-3}$ ]{\includegraphics[width= .475\textwidth]{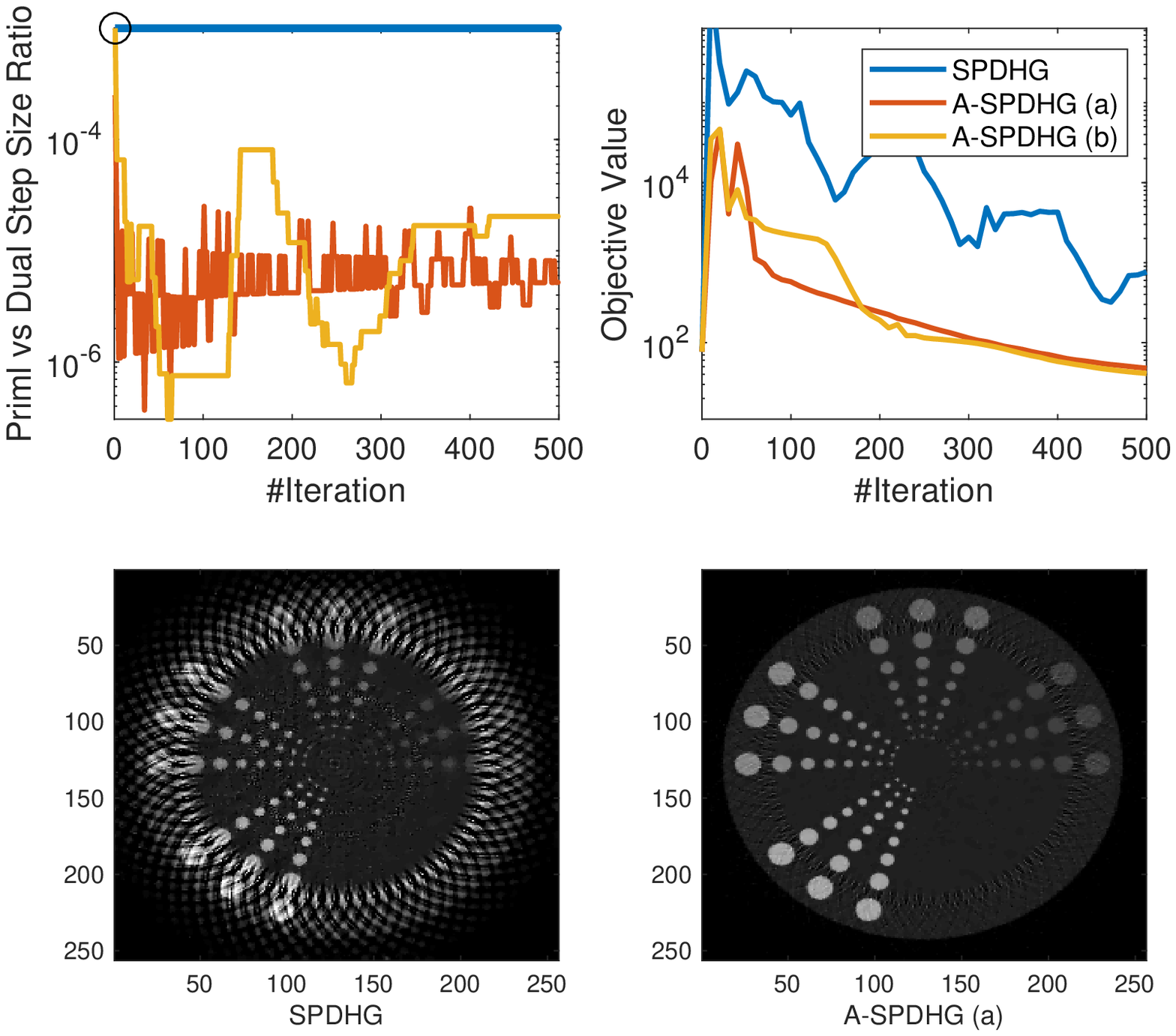}}
    \subfloat[starting ratio $10^{-5}$ ]{\includegraphics[width= .475\textwidth]{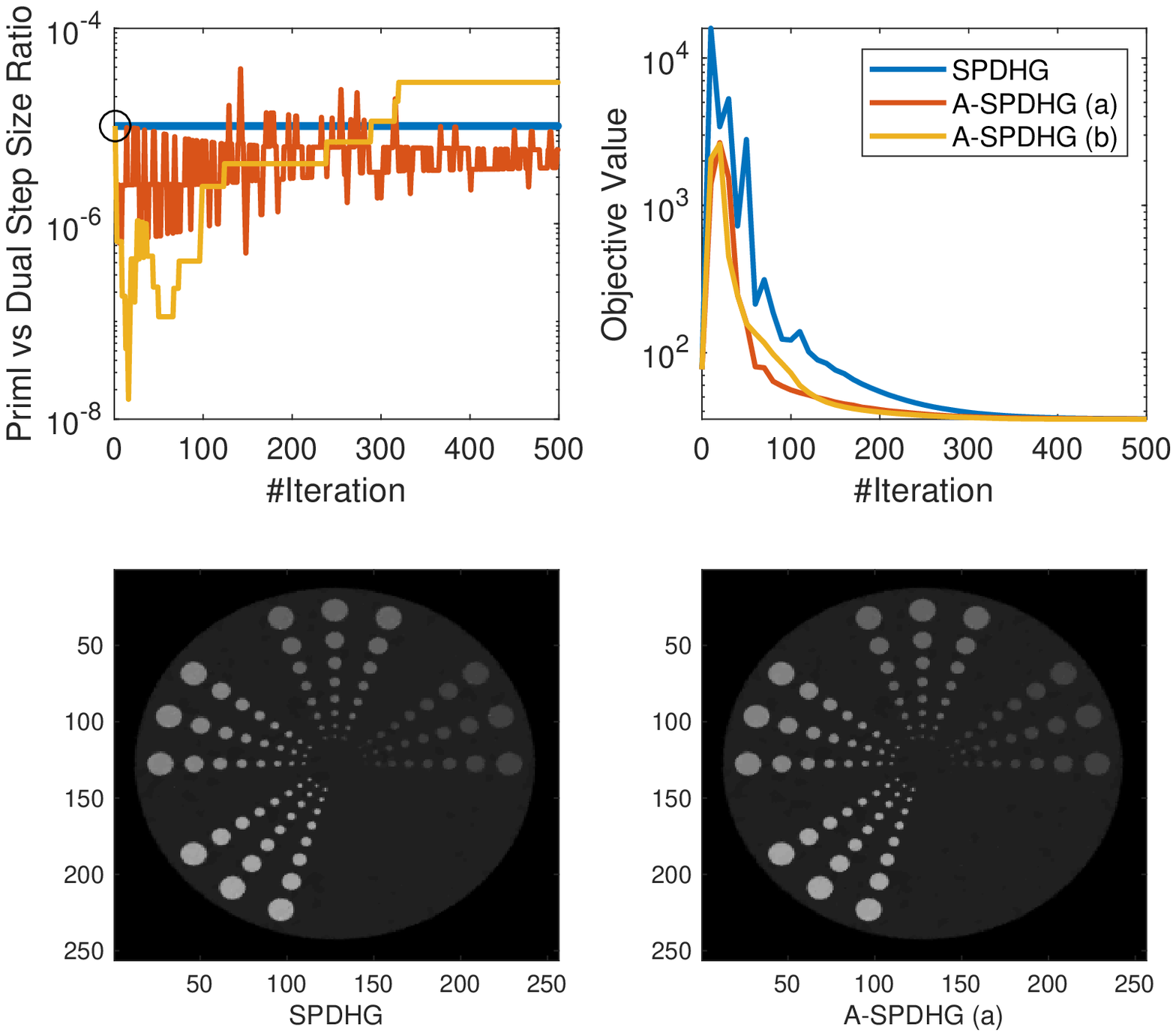}}\\
      \subfloat[starting ratio $10^{-7}$ ]{\includegraphics[width= .475\textwidth]{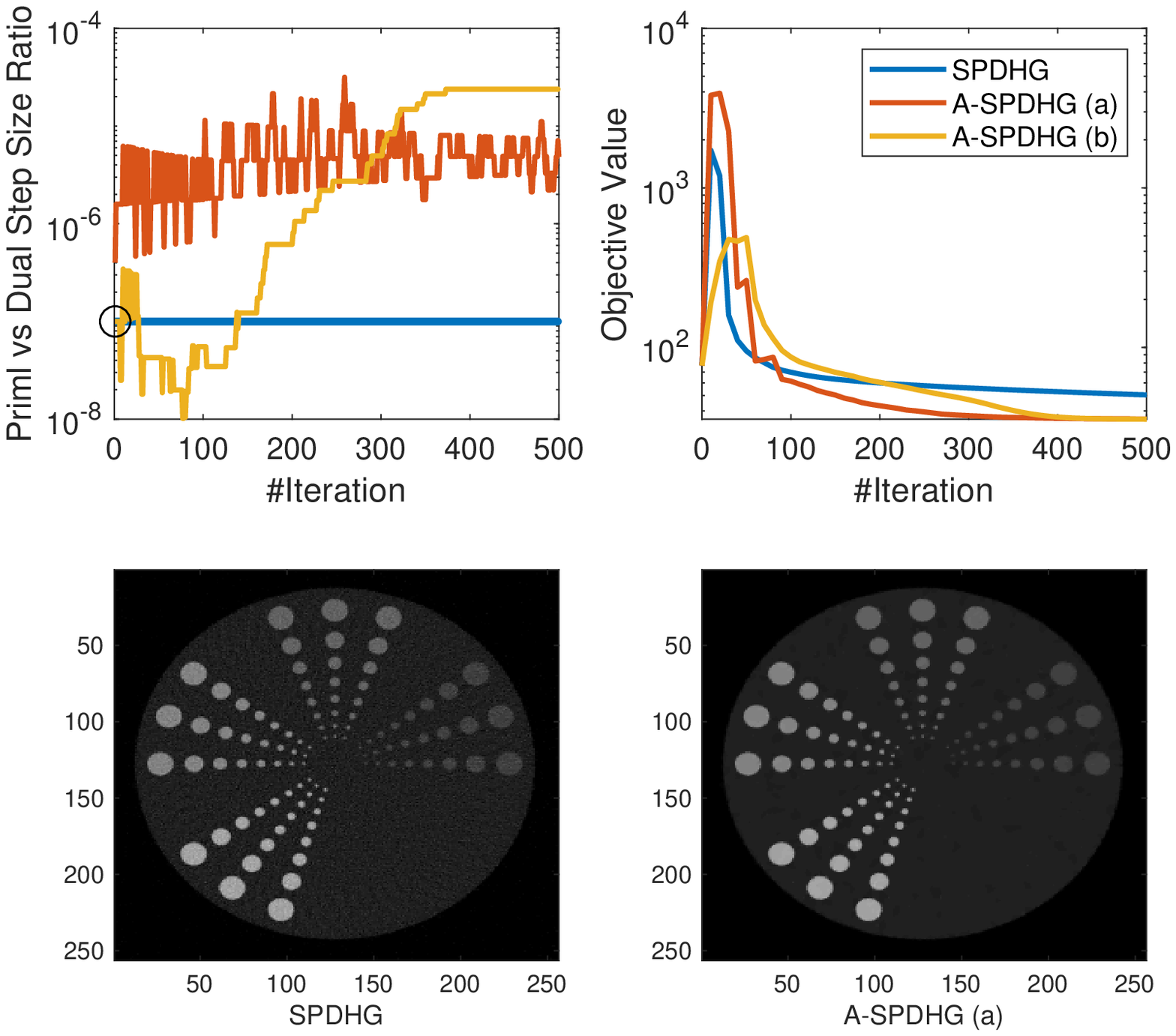}}
      \subfloat[starting ratio $10^{-9}$ ]{\includegraphics[width= .475\textwidth]{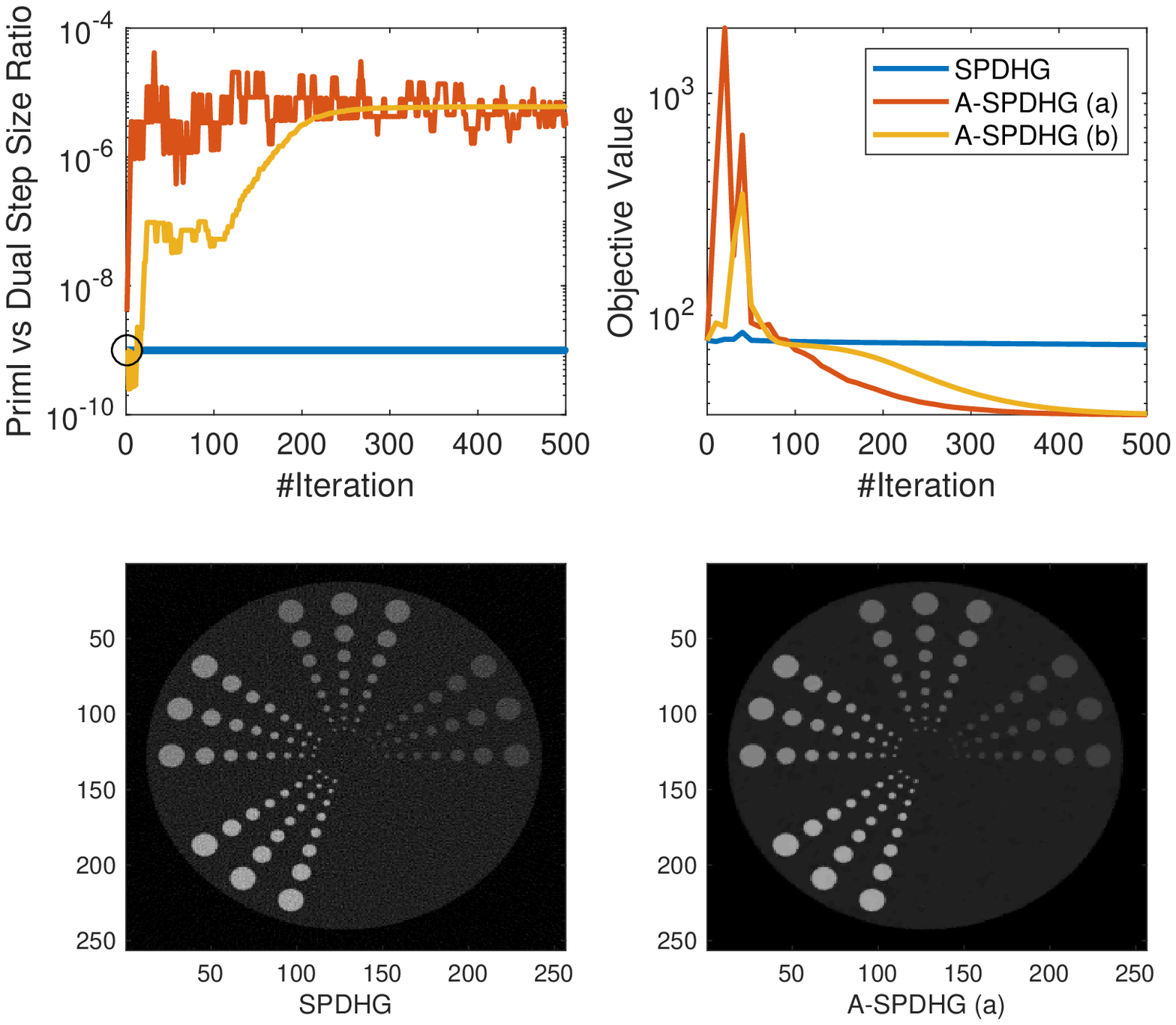}}\\
   \caption{Comparison between SPDHG and A-SPDHG \TrackChange{with warm-start using a FBP (filtered back-projection) on Low-Dose CT. Comparing to the results shown in the Figure \ref{f2b} which are without warm-start, actually our methods seem to compare even more favorably with warm-start. Please also note that the early jump in terms of function value is within our expectation due to the stochasticity of the algorithms.} We include the images reconstructed by the algorithms at termination (50th epoch).}
\label{f2ws}
\end{figure}

\clearpage

\begin{figure}[t]
   \centering

      \subfloat[starting ratio $10^{-3}$ ]{\includegraphics[width= .475\textwidth]{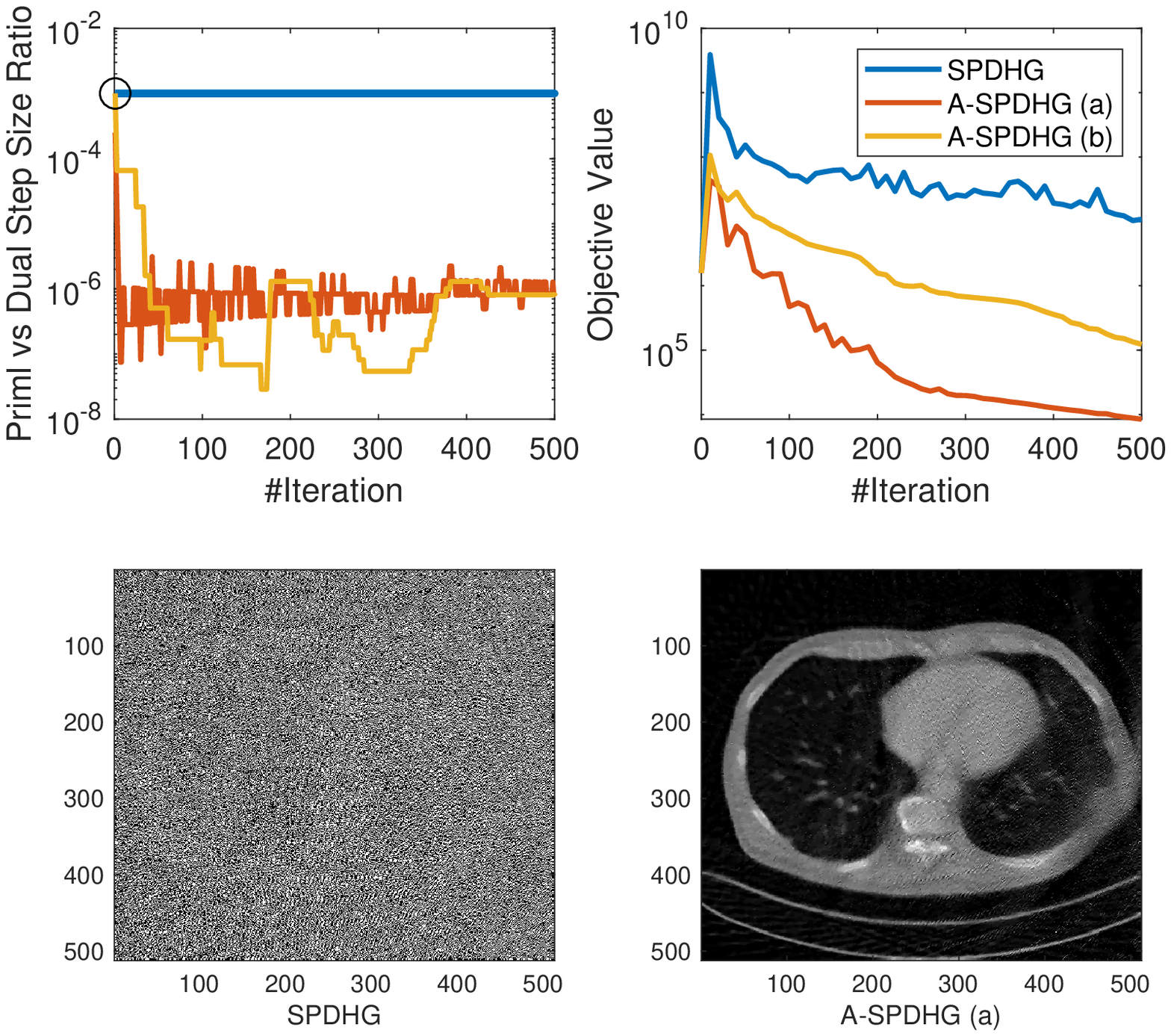}}
    \subfloat[starting ratio $10^{-5}$ ]{\includegraphics[width= .475\textwidth]{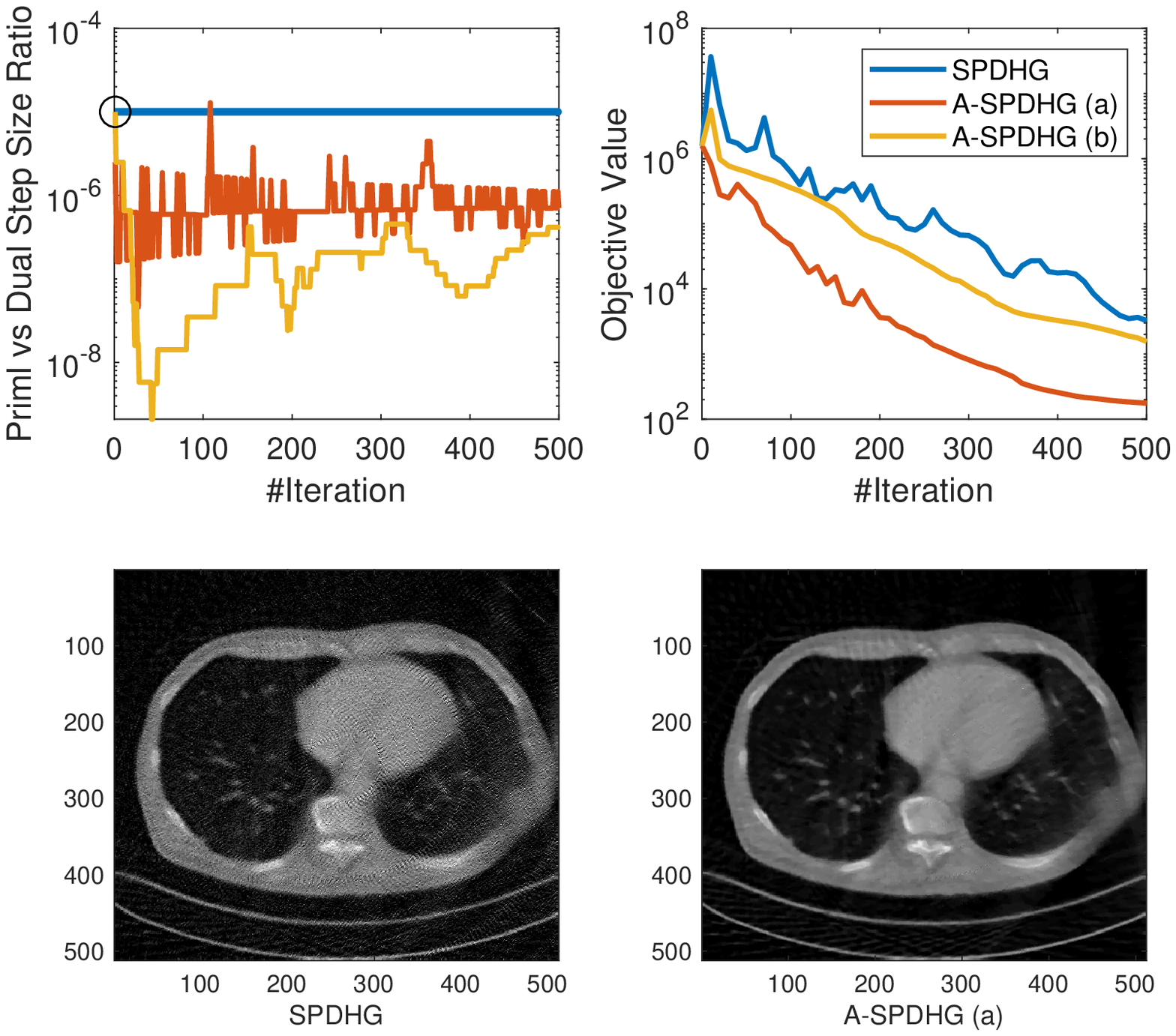}}\\
      \subfloat[starting ratio $10^{-7}$ ]{\includegraphics[width= .475\textwidth]{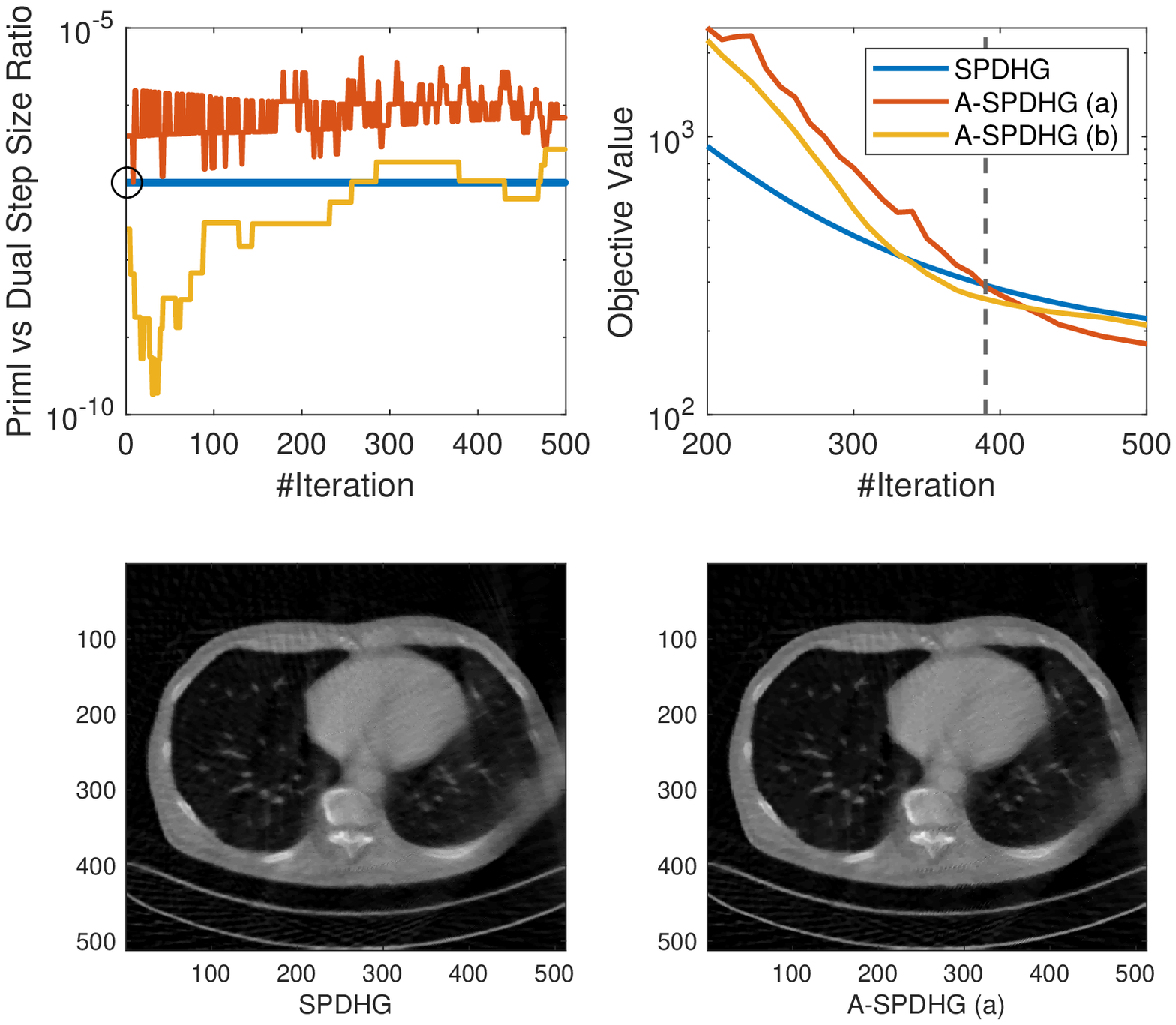}}
          \subfloat[starting ratio $10^{-9}$ ]{\includegraphics[width= .475\textwidth]{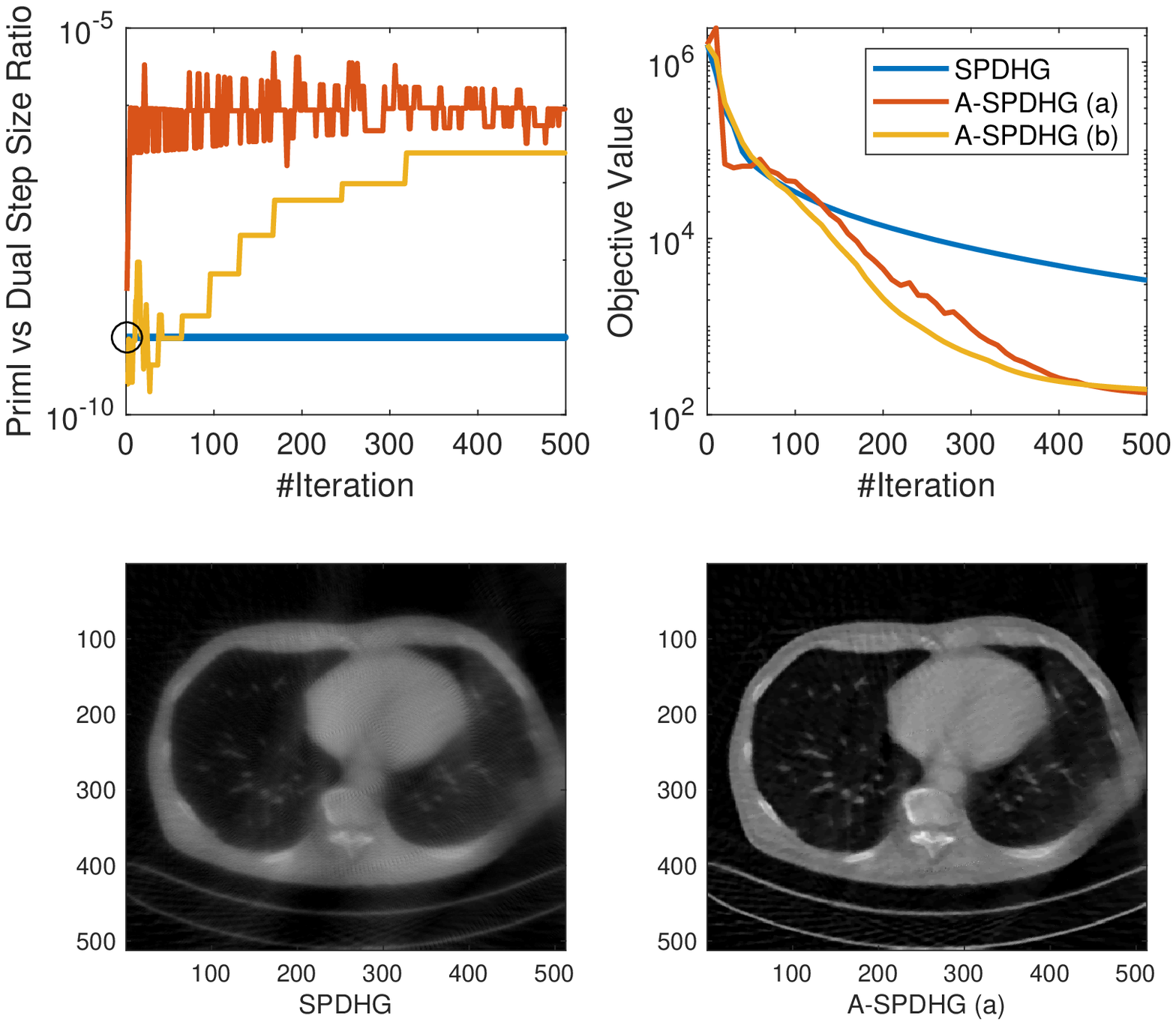}}
   \caption{Comparison between SPDHG and A-SPDHG on Limited-Angle CT (Example 2), with a variety of starting primal-dual step-size ratios. Here the forward operator $A\in\mathbb{R}^{m\times d}$ where the dimension $m=92160$, $d= 262144$. We include the images reconstructed by the algorithms at termination (50th epoch).}
\label{f2c}
\end{figure}
\clearpage

\section{Numerical Experiments}
\label{sec:Numerics}
In this section we present numerical studies of the proposed scheme in solving one of the most typical imaging inverse problems, the Computed Tomography (CT). We compare A-SPDHG algorithm with the original SPDHG, on different choices of starting ratio of the primal and dual step-sizes.

In our CT imaging example, we seek to reconstruct the tomography images from fan-beam X-ray measurement data, by solving the following TV-regularized objective:
\begin{equation}
   x^\star \in \arg\min_{x \in \TrackChange{\mathbb{R}^d}} \frac{1}{2}\|Ax - b\|_2^2 + \lambda\|Dx\|_1
\end{equation}
where $D$ denotes the 2D differential operator, $A\in\mathbb{R}^{m\times d}$ and $x\in\mathbb{R}^{d}$. We consider three fanbeam CT imaging modalities: Sparse-View CT, Low-Dose CT and Limited-Angle CT. We test the A-SPDHG and SPDHG on two images of different sizes (Example 1 on a phantom image sized $1024 \times 1024$, while Example 2 being \TrackChange{an image from the Mayo Clinic Dataset \cite{mccollough2016tu} sized $512 \times 512$.}), on 4 different starting ratios ($10^{-3}$, $10^{-5}$, $10^{-7}$ and $10^{-9}$). \TrackChange{We interleave partitioned the measurement data and operator into $n=10$ minibatches for both algorithms. To be more specific, we first collect all the X-ray measurement data and list them consecutively from 0 degree to 360 degree to form the full $A$ and $b$, and then interleavingly group every 10-th of the measurements into one minibatch, to form the partition $\{A_i\}_{i=1}^{10}$ and $\{b_i\}_{i=1}^{10}$.}

For A-SPDHG we choose to use the approximation step for $d^k$ presented in \eqref{approx_d} with $10\%$ subsampling hence the computational overhead is negligible in this experiment. We initialize all algorithms from a zero-image.

We present our numerical results in Figures \ref{f1}, \ref{f2}, \ref{f2b} and \ref{f2c}. In these plots we compare the convergence rates of the algorithms in terms of number of iterations (the execution time per iteration for the algorithms are almost the same, as the overhead of A-SPDHG is trivial numerically). Among these, Figures \ref{f1} and \ref{f2} report the results for large-scale sparse-view CT experiments on a phantom image and a lung CT image from Mayo-Clinic dataset \cite{mccollough2016tu}, while Figure \ref{f2b} reports the results for low-dose CT experiments where we simulate a large number of measurements corrupted with a significant amount Poisson noise, and then, in Figure \ref{f2c} \TrackChange{we report the results for limited-angle CT which only a range of $0$-degree to $150$-degree of measurement angles are present, while the measurements from the rest $[150, 360]$ degrees of angles are all missing.} In all these examples we can consistently observe that no matter how we initialize the primal-dual step-size ratio, A-SPDHG can automatically and consistently adjust the step size ratio to the optimal choice which is around either $10^{-5}$ or $10^{-7}$ for these four different CT problems, and significantly outperform the vanilla SPDHG for the cases where the starting ratio is away from the optimal range. Meanwhile, even for the cases where the starting ratio of SPDHG algorithm is near-optimal, we can observe consistently from most of these examples that our scheme outperforms the vanilla SPDHG algorithm locally after a certain number of iterations (highlighted by the vertical dash lines in relevant subfigures), which further indicates the benefit of adaptivity for this class of algorithms\footnote{The most typical example here would be the Figure \ref{f1}(b) where the optimal step-size ratio selected by the adaptive scheme at convergence is almost exactly $10^{-5}$, where we have set SPDHG to run with this ratio. We can still observe benefit of local convergence acceleration given by our adaptive scheme.}. Note that throughout all these different examples, we use only one fixed set of parameters for A-SPDHG suggested in the previous section, which again indicates the strong practicality of our scheme.

\TrackChange{
For the low-dose CT example, we run two extra sets of experiments, regarding a larger number of partioning of minibatches (40) on Figure \ref{f2_40}, and warm-start from a better initialization image obtained via filter-backprojection on Figure \ref{f2ws}. We found that in all these extra examples we consistently observe superior performances of A-SPDHG over the vanilla SPDHG especially when the primal-dual step-size ratios are suboptimal. Interestingly, we found that the warm-start's effect does not have noticeable impact of the comparative performances between SPDHG and A-SPDHG. This is mainly due to the fact that the SPDHG with suboptimal primal-dual step-size ratio will converge very slowly in high accuracy regimes (see Fig \ref{f2ws}(d) for example) in practice hence the warm-start won't help much here.
}

We should also note that conceptually all the hyperparameters in our adaptive schemes are basically the controllers of the adaptivity of the algorithm (while for extreme choices we recover the vanilla SPDHG). In Figures \ref{f3} and \ref{f4}, we present some numerical studies on the choices of hyperparameters of rule (a) and rule (b) of A-SPDHG algorithm. We choose the fixed starting ratio of $10^{-7}$ for primal-dual step-sizes in these experiments. For rule (a), we found that it is robust to the choice of the starting shrinking rate $\alpha_0$, shrinking speed $\eta$ and the gap $\delta$. Overall, we found that these parameters have weak impact of the convergence performance of our rule (a) and easy to choose.

For rule (b), we found that the performance is more sensitive to the choice of parameter $c$ and $\eta$ comparing to rule (a), although the dependence is still weak. \TrackChange{Our numerical studies suggest that rule (a) is a better-performing choice than rule (b), but each of them have certain mild weaknesses (the first rule has a slight computational overhead which can be partitially addressed with subsampling scheme, while the second rule seems often being slower than the first rule), which require further studies and improvements.} Nevertheless, we need to emphasis that all these parameters are essentially controlling the degree of adaptivity of the algorithms and fairly easy to choose, noting that for all these CT experiments with varying sizes/dimensions and modalities we only use one fixed set of the hyperparameters in A-SPDHG, and we are already able to consistently observe numerical improvements over vanilla SPDHG.

\section{Conclusion}
\label{sec:conclusion}

In this work we propose a new framework (A-SPDHG) for adaptive step-size balancing in stochastic primal-dual hybrid gradient methods. We first derive theoretically sufficient conditions on the adaptive primal and dual step-sizes for ensuring convergence in the stochastic setting. We then propose a number of practical schemes which satisfy the condition for convergence, and our numerical results on imaging inverse problems supports the effectiveness of the proposed approach.

To our knowledge, this work constitutes the first theoretical analysis of adaptive step-sizes for a stochastic primal-dual algorithm. Our on-going work includes the theoretical analysis and algorithmic design of further accelerated stochastic primal-dual methods with line-search schemes for even faster convergence rates.

\begin{figure}[t]
   \centering

      \subfloat[Test on the choices for $\alpha_0$ ]{\includegraphics[width= .3\textwidth]{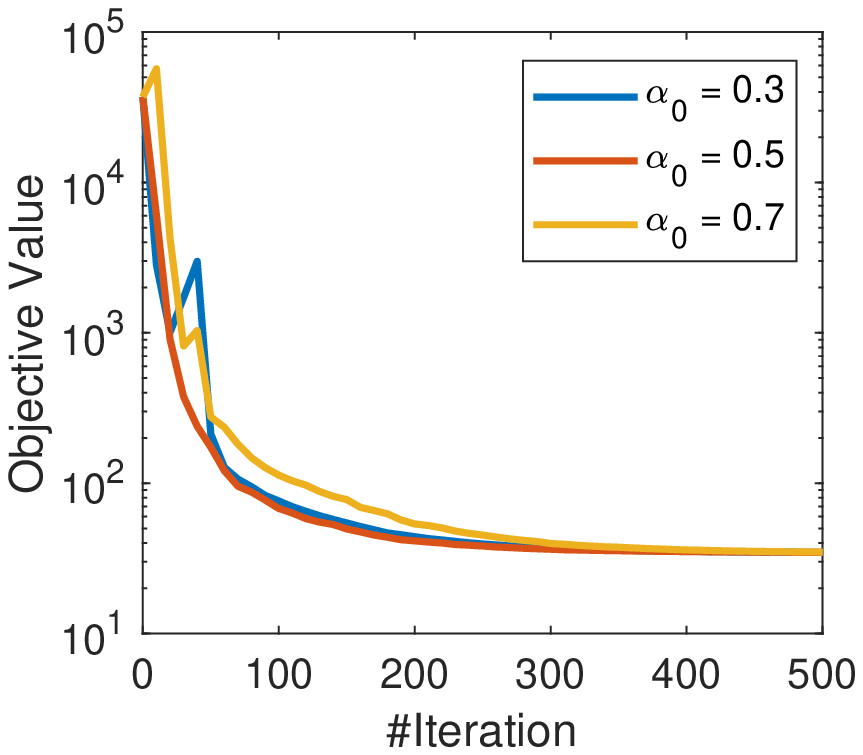}}
      \subfloat[Test on the choices for $\eta$ ]{\includegraphics[width= .3\textwidth]{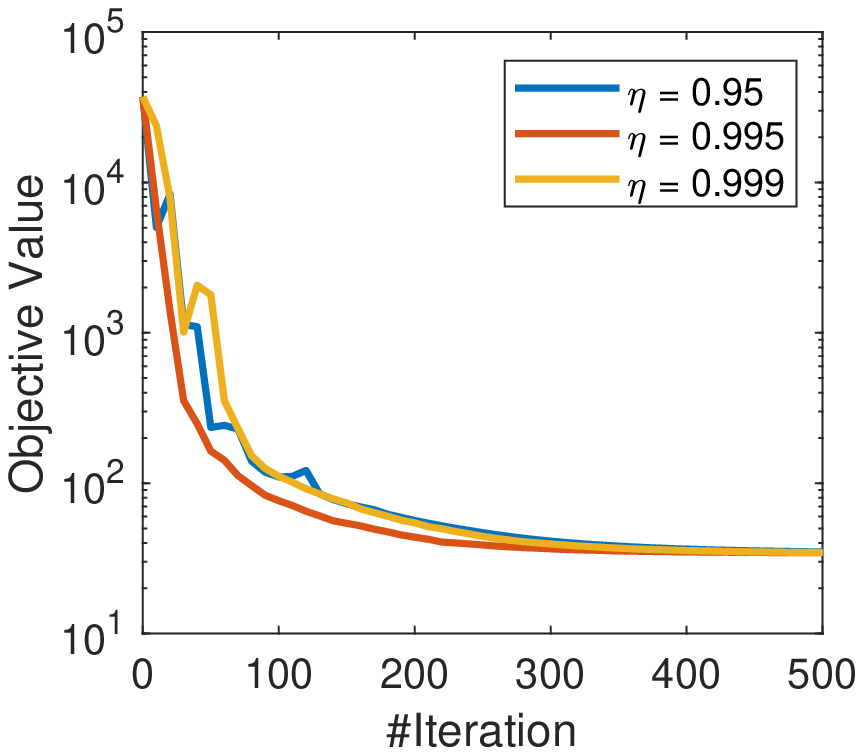}}
      \subfloat[Test on the choices for $\delta$ ]{\includegraphics[width= .3\textwidth]{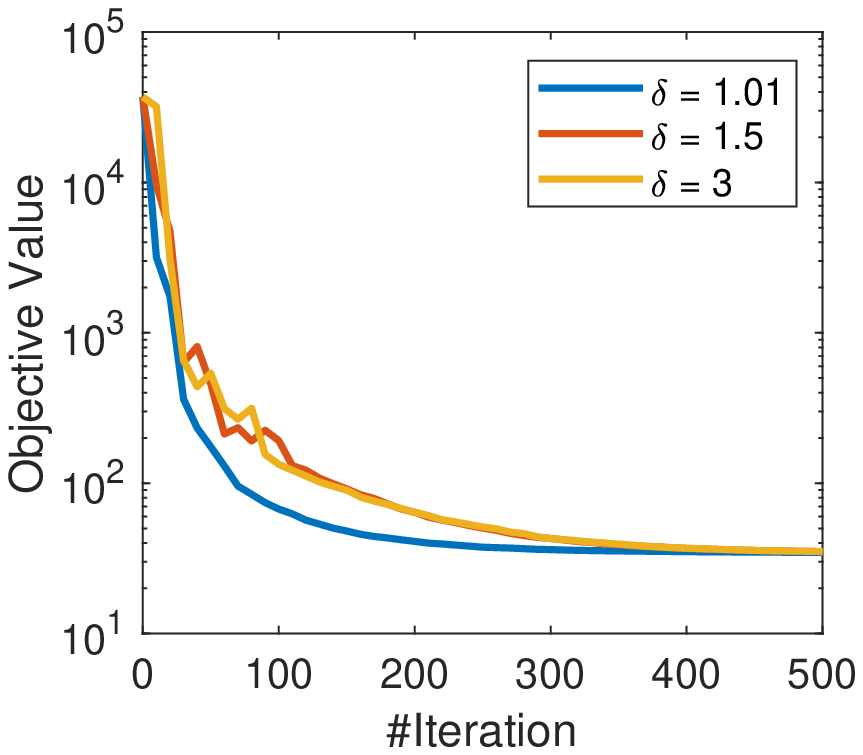}}
   \caption{Test on different choices of parameters of A-SPDHG (rule-a) on X-ray Low-Dose fanbeam CT example, starting ratio of primal-dual step-sizes: $10^{-7}$. We can observe that the performance of ASPDHG has only minor dependence on these parameter choices.}
\label{f3}
\end{figure}

\begin{figure}[t]
   \centering

      {\includegraphics[width= .3\textwidth]{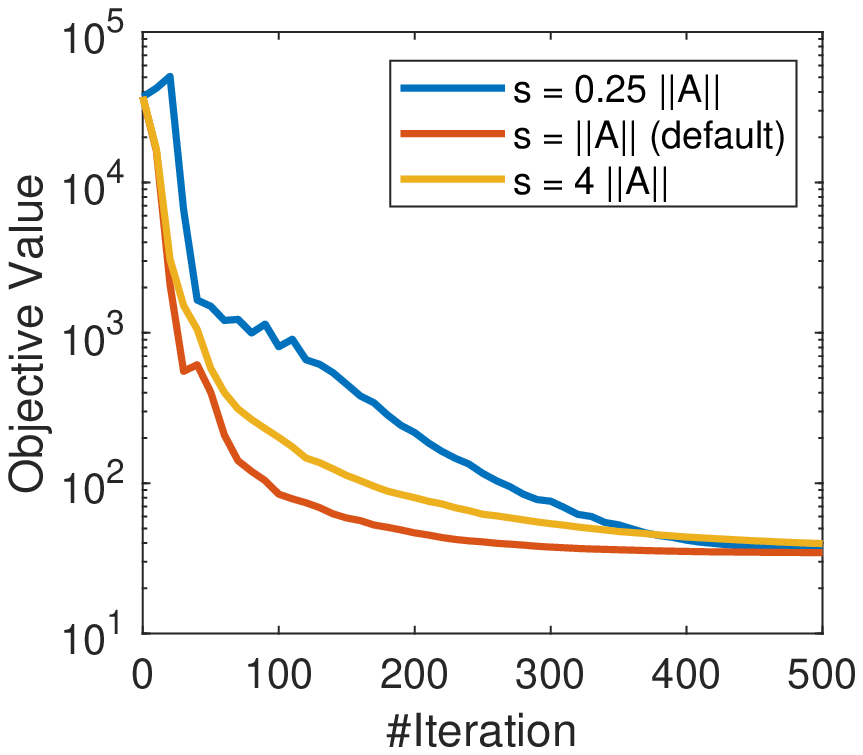}}
     {\includegraphics[width= .3\textwidth]{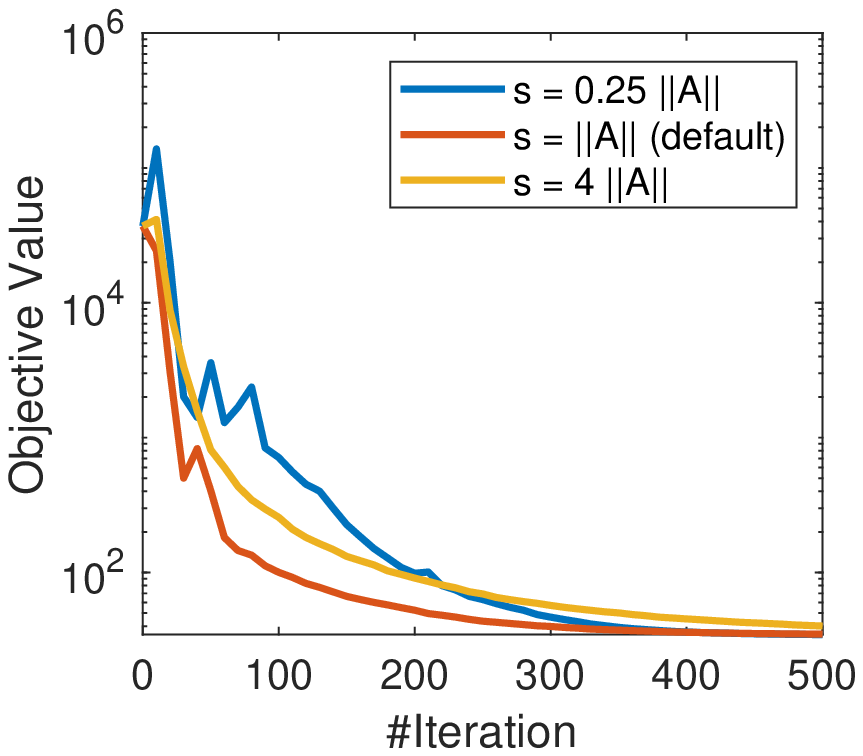}}
   \caption{Test on the default choice $s = \|A\|$  of A-SPDHG (rule-a) on X-ray Low-Dose fanbeam CT example. Left figure: starting ratio of primal-dual step-sizes: $10^{-7}$. Right figure: starting ratio of primal-dual step-sizes: $10^{-5}$. We can observe that our default choice of $s$ is indeed a reasonable choice (at least near-optimal) in practice, and when deviating from it may lead to slower convergence.}
\label{f3a}
\end{figure}

\begin{figure}[t]
   \centering

      \subfloat[Test on the choices for $c$ ]{\includegraphics[width= .3\textwidth]{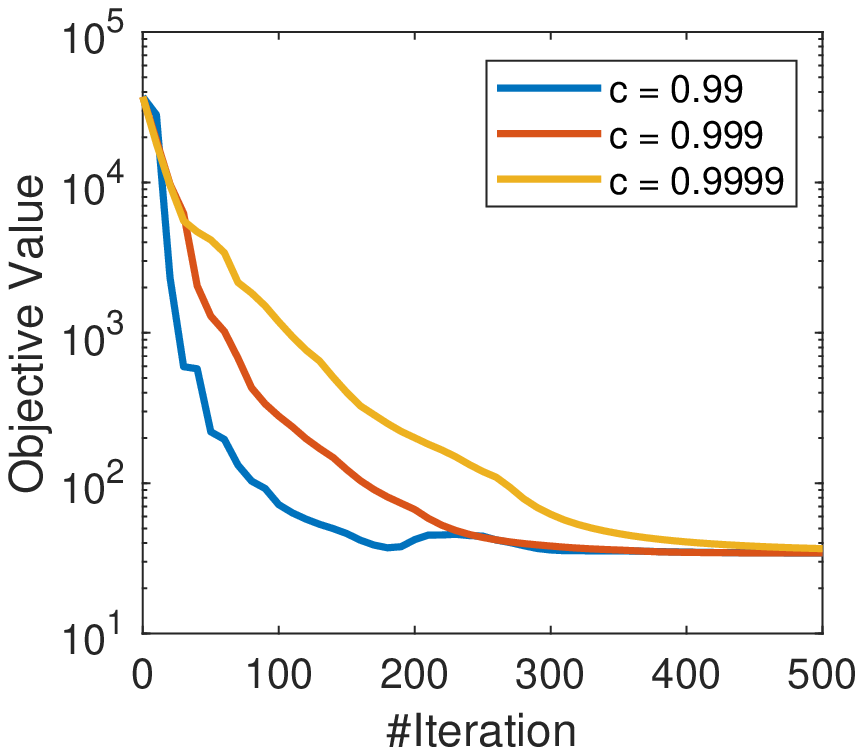}}
      \subfloat[Test on the choices for $\eta$ ]{\includegraphics[width= .3\textwidth]{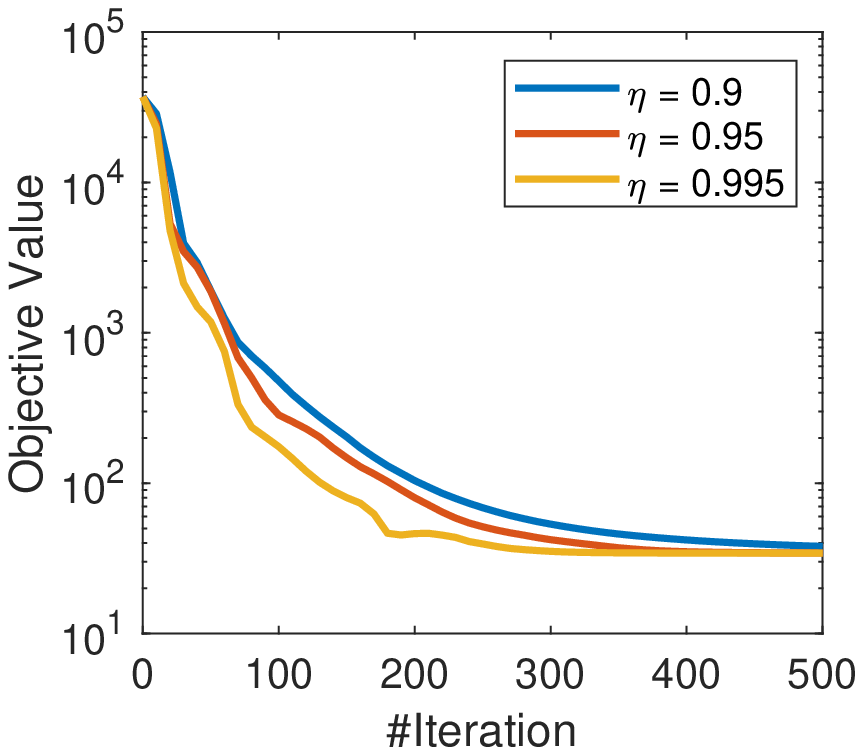}}
   \caption{Test on different choices of parameters of A-SPDHG (rule-b) on X-ray Low-Dose fanbeam CT example, starting ratio of primal-dual step-sizes: $10^{-7}$.}
\label{f4}
\end{figure}

\section{Complementary material for Section \ref{sec:theory}}
\label{sec:complements}

We begin by a useful lemma.

\begin{lemma}
Let $a,\,b$ be positive scalars, $\beta\in(0,1)$, and $P$ a bounded linear operator from a Hilbert space $X$ to a Hilbert space $Y$. Then,
\begin{align}
 (ab)^{-1/2}\|P\| \leq 1    &\quad\Leftrightarrow\quad  
\begin{pmatrix}
a\, \text{Id} & \TrackChange{P^*} \\
\TrackChange{P} & b \,\text{Id}
\end{pmatrix} \succcurlyeq 0.\label{eq:A2C_a}\\
 (ab)^{-1/2}\|P\| \leq \beta  &\quad\Leftrightarrow\quad  
\begin{pmatrix}
a \,\text{Id} & \TrackChange{P^*}\\
\TrackChange{P} & b \,\text{Id}
\end{pmatrix} \succcurlyeq  (1-\beta) \begin{pmatrix}
a \,\text{Id} & 0\\
0 & b \,\text{Id}
\end{pmatrix}.
\label{eq:A2C_b}
\end{align}
\label{lm:Abstract2Concrete}
\end{lemma}

\begin{proof}
Let us call
\begin{align*}
M &=\begin{pmatrix}
a \,\text{Id} & \TrackChange{P^*}\\
\TrackChange{P} & b \,\text{Id}
\end{pmatrix}.
\end{align*}
For all $(x,y)\in X\times Y$,
\begin{align*}
\|(x,y)\|_M^2 &\geq a\|x\|^2 + b \|y\|^2 - 2 \|P\| \|x\| \|y\| = \|x\|_a^2 +  \|y\|_b^2 - 2 (ab)^{-1/2}\|P\| \|x\|_a\|y\|_b,
\end{align*}
which proves the direct implication of \eqref{eq:A2C_a}. \TrackChange{For the converse implication, consider $x\in X \setminus \left\{ 0\right\}$ such that $\|Px\|=\|P\| \|x\|$ and $y=-\lambda Px$ for a scalar $\lambda$. Then, the non-negativity of the polynomial 
$$\frac{\|(x,y)\|_M^2}{\|x\|^2} = b\|P\|^2 \lambda^2 - 2 \|P\|^2 \lambda + a$$
for all $\lambda \in \mathbb{R}$ implies that $\|P\|^4 - ab \|P\|^2 \leq 0$, which is equivalent to the desired conclusion $(ab)^{-1/2}\|P\| \leq 1$.}\newline
Equivalence \eqref{eq:A2C_b} is straightforward by noticing that
\begin{align*}
\begin{pmatrix}
a \,\text{Id} & \TrackChange{P^*}\\
\TrackChange{P} & b \,\text{Id}
\end{pmatrix} \succcurlyeq  (1-\beta) \begin{pmatrix}
a \,\text{Id} & 0\\
0 & b \,\text{Id}
\end{pmatrix} \Leftrightarrow \begin{pmatrix}
\beta a \,\text{Id} & \TrackChange{P^*}\\
\TrackChange{P} & \beta b \,\text{Id}
\end{pmatrix} \succcurlyeq  0.
\end{align*}
\end{proof}

Let us now turn to the proof of Lemma \ref{lm:PrimalDualBalancing}.

\begin{proof}[Proof of Lemma \ref{lm:PrimalDualBalancing}]
Let us assume that the step-sizes satisfy the assumptions of the lemma. Then, Assumption (i) of Theorem \ref{thm:A-SPDHG} is straightforwardly satisfied. Moreover, for $i\in \llbracket1,n\rrbracket$, the product sequence $(\tau^k \sigma_i^k)_{k \in \mathbb{N}}$ is constant along the iterations by equation \eqref{eq:UpdateRule} and satisfies equation \eqref{eq:A-ConvergenceCondition} for iterate $k=0$, thus satisfies \eqref{eq:A-ConvergenceCondition} for all $k \in \mathbb{N}$ \TrackChange{ for $\beta = \max_i\left\{\tau^0 \sigma_i^0\|A_i\|^2 / p_i\right\}$}, which proves Assumption (ii). Finally, equation \eqref{eq:PrimalDualBalancing} implies that Assumption (iii) is satisfied.
\end{proof}

\section{Declarations}
\subsection{Ethical approval}
This declaration is not applicable.
\subsection{Competing interests}
There are no competing interests to declare.
\subsection{Authors' contribution}
CD, MJE and AC elaborated the proof strategy and CD wrote parts 1, 2 and 6. JT worked on the algorithmic design, performed the numerical experiments and wrote parts 3-5. All authors reviewed the manuscript.
\subsection{Funding}
CD acknowledges support from the EPSRC (EP/S026045/1). MJE acknowledges support from the EPSRC (EP/S026045/1, EP/T026693/1, EP/V026259/1) and the Leverhulme Trust (ECF-2019-478). CBS acknowledges support from the Philip Leverhulme Prize,
the Royal Society Wolfson Fellowship, the EPSRC advanced career fellowship EP/V029428/1, EPSRC grants
EP/S026045/1 and EP/T003553/1, EP/N014588/1, EP/T017961/1, the Wellcome Innovator Awards 215733/Z/19/Z
and 221633/Z/20/Z, the European Union Horizon 2020 research and innovation programme under the Marie
Skodowska-Curie grant agreement No. 777826 NoMADS, the Cantab Capital Institute for the Mathematics of
Information and the Alan Turing Institute.
\subsection{Availability of data and materials}

The related implementation of the algorithms and the image data used in the experiment will be made available on the website \url{https://junqitang.com} . For the phantom image example we use the one in the experimental section of \cite{SPDHG}, while for the lung CT image example we use an image from the Mayo Clinic Dataset \cite{mccollough2016tu} which is publicly available.

\bibliographystyle{plain} 

\bibliography{biblio}

\end{document}